\documentclass[11pt,twoside]{article}
\usepackage{vmargin}
\usepackage{amsthm}
\usepackage{amsmath}
\setmargins{32mm}{20mm}{14.6cm}{22cm}{1cm}{1cm}{1cm}{1cm}
\usepackage{latexsym}

\usepackage{amssymb}
\usepackage{stmaryrd}
\usepackage{mathrsfs}
\usepackage[all]{xy}
\usepackage{xr}

\newtheorem{theorem}{Theorem}[section]
\newtheorem{proposition}[theorem]{Proposition}
\newtheorem{corollary}[theorem]{Corollary}

\newtheorem{lemma}[theorem]{Lemma}
\newtheorem*{theorem*}{Theorem}
\newtheorem*{proposition*}{Proposition}
\newtheorem*{corollary*}{Corollary}
\newtheorem*{lemma*}{Lemma}
\newtheorem*{conjecture*}{Conjecture}

\theoremstyle{definition}
\newtheorem{definition}[theorem]{Definition}

\newtheorem*{definition*}{Definition}

\theoremstyle{remark}
\newtheorem{example}[theorem]{Example}

\newtheorem{remark}[theorem]{Remark}
\newtheorem{remarks}[theorem]{Remarks}

\newtheorem*{example*}{Example}
\newtheorem*{examples*}{Examples}
\newtheorem*{remark*}{Remark}
\newtheorem*{remarks*}{Remarks}
\newtheorem*{exercise*}{Exercise}

\numberwithin{equation}{section}

\newcommand\da{\!\downarrow\!}
\newcommand\ra{\rightarrow}

\newcommand\lra{\longrightarrow}
\newcommand\lla{\longleftarrow}
\newcommand\id{\mathrm{id}}

\newcommand\ten{\otimes}

\newcommand\vareps{\varepsilon}
\newcommand\eps{\epsilon}

\renewcommand\H{\mathrm{H}}
\newcommand\z{\mathrm{Z}}

\newcommand\N{\mathbb{N}}
\newcommand\Z{\mathbb{Z}}

\newcommand\bD{\mathbb{D}}

\newcommand\bL{\mathbb{L}}

\newcommand\bR{\mathbb{R}}
\newcommand\bS{\mathbb{S}}

\newcommand\C{\mathcal{C}}

\newcommand\cA{\mathcal{A}}

\newcommand\cD{\mathcal{D}}

\newcommand\cG{\mathcal{G}}
\newcommand\cH{\mathcal{H}}

\newcommand\cL{\mathcal{L}}
\newcommand\cM{\mathcal{M}}

\newcommand\cP{\mathcal{P}}

\newcommand\cS{\mathcal{S}}

\newcommand\V{\mathscr{V}}

\newcommand\sA{\mathscr{A}}

\newcommand\sE{\mathscr{E}}

\newcommand\sR{\mathscr{R}}

\newcommand\fL{\mathfrak{L}}

\newcommand\fS{\mathfrak{S}}

\renewcommand\L{\Lambda}

\newcommand\m{\mathfrak{m}}

\newcommand\Ho{\mathrm{Ho}}

\newcommand\Alg{\mathrm{Alg}}

\newcommand\Mod{\mathrm{Mod}}

\newcommand\Hom{\mathrm{Hom}}
\newcommand\Map{\mathrm{Map}}

\newcommand\HHom{\underline{\mathrm{Hom}}}

\newcommand\im{\mathrm{Im\,}}

\newcommand\Gp{\mathrm{Gp}}

\newcommand\Spec{\mathrm{Spec}\,}

\newcommand\Spf{\mathrm{Spf}\,}
\newcommand\Art{\mathrm{Art}}

\newcommand\Set{\mathrm{Set}}

\newcommand\Sp{\mathrm{Sp}}

\newcommand\FD{\mathrm{FD}}

\newcommand\Lim{\varprojlim}
\newcommand\LLim{\varinjlim}
\DeclareMathOperator*{\holim}{holim}

\newcommand\into{\hookrightarrow}
\newcommand\onto{\twoheadrightarrow}
\newcommand\abuts{\implies}
\newcommand\xra{\xrightarrow}
\newcommand\xla{\xleftarrow}

\newcommand\bt{\bullet}
\newcommand\by{\times}

\newcommand\mc{\mathrm{MC}}

\newcommand\ddef{\mathrm{Def}}

\newcommand\Vect{\mathrm{Vect}}

\newcommand\Symm{\mathrm{Symm}}

\newcommand\Tot{\mathrm{Tot}\,}
\newcommand\diag{\mathrm{diag}\,}

\newcommand\ev{\mathrm{ev}}
\newcommand\ind{\mathrm{ind}}
\newcommand\pro{\mathrm{pro}}

\newcommand\pd{\partial}

\newcommand\half{\frac{1}{2}}

\newcommand\Gr{\mathrm{Gr}}

\newcommand\LA{\mathrm{LA}}

\newcommand\sk{\mathrm{sk}}
\newcommand\cosk{\mathrm{cosk}}

\newcommand\op{\mathrm{opp}}
\newcommand\opp{\mathrm{opp}}

\newcommand\co{\colon\thinspace}

\newcommand\uleft\underleftarrow
\newcommand\uline\underline
\newcommand\uright\underrightarrow

\begin{document}
\title{Unifying derived deformation theories}
\author{J. P. Pridham
\thanks{  This work was supported by Trinity College, Cambridge; and by the Engineering and Physical Sciences Research Council [grant number  EP/F043570/1]. 
}
}
\date{}
\maketitle

\begin{abstract}
We develop a framework for derived deformation theory, valid in all characteristics. This gives a model category reconciling local and global approaches to derived moduli theory. In characteristic $0$, we use this to show that the homotopy categories of DGLAs and SHLAs ($L_{\infty}$-algebras) considered by Kontsevich, Hinich and Manetti are equivalent, and are compatible with the derived stacks of To\"en--Vezzosi and Lurie. Another application is that the cohomology groups associated to any classical deformation problem (in any characteristic)   admit the same operations as Andr\'e-Quillen cohomology. 
\end{abstract}

\tableofcontents

\section*{Introduction}

There are two main approaches to derived moduli theory. The local approach of \cite{Kon}, \cite{Man2} and \cite{hinstack} uses DGLAs and SHLAs to yield derived deformation functors for a very wide range of problems, but is restricted to characteristic zero, with its constructions seldom extending to global problems. By contrast, the derived moduli stacks
of \cite{hag2} and \cite{lurie} give a global formulation, valid in all characteristics, but  have only been constructed for a comparatively narrow class of examples. In this paper, we bridge the gap between the two approaches, 
as explained in Proposition \ref{seattle}.

In \cite{lurie}, Lurie defines a derived stack as a functor from topological rings to topological spaces, or equivalently from simplicial rings to simplicial sets.  As we are only studying infinitesimal deformations, our functors are instead defined on  Artinian simplicial rings.   The classical deformation groupoid will then be the fundamental groupoid of this functor, restricted to rings (rather than simplicial rings). \cite{ddt2}  shows how to define such functors for all classical deformation problems.

Section \ref{one} contains definitions and basic properties of functors of this form. The crucial new ingredient is    a property of functors $F$ which we call quasi-smoothness\footnote{This terminology is adapted from earlier deformation theory literature; unfortunately, ``quasi-smooth'' nowadays is almost always used to mean virtually LCI.}; this means that $F$ maps small extensions to fibrations, and acyclic small extensions to trivial fibrations. This is partly  motivated by noting that an $\infty$-hypergroupoid  is just a fibrant simplicial set (\cite{duskin}). For any such functor, we can define cohomology groups $\H^i(F)$, for $i \in \Z$, and there are long exact sequences in which these groups simultaneously play the r\^oles of tangent and obstruction spaces (Theorem \ref{robs}). Thus quasi-smoothness   captures the flavour of $\infty$-geometricity considered in \cite{hag2} and \cite{lurie}, without the drawbacks of an inductive construction.

Rather than embedding the geometric stacks in a larger model structure (as for the $D^-$-stacks of  \cite{hag2}), we have  a model category all of whose objects are geometric: in Section \ref{model}, we show how to put a model structure on the category of all left-exact functors from Artinian simplicial rings to simplicial sets. In this model structure, the fibrations are precisely the quasi-smooth maps, so each equivalence class has a quasi-smooth representative. There are analogues of Eilenberg-Maclane spaces for representing cohomology groups, and every weak equivalence class has a unique minimal model. The homotopy category satisfies a  Brown-type representability  property  (Theorem \ref{schrep}) analogous to Schlessinger's Theorem.\footnote{Functors satisfying the property were later dubbed \emph{formal moduli problems} by Lurie.} 

Section \ref{back} provides a summary of existing approaches to derived deformations: Manetti's extended functors, Hinich's formal stacks, and the derived stacks of To\"en--Vezzosi and Lurie. The only new result is
Proposition \ref{toencomp}, which  shows how our geometric stacks  may be regarded as germs of geometric $D^-$-stacks.

Section \ref{alternative} compares the homotopy category of Section \ref{model} with established homotopy categories used to study derived deformations in characteristic zero. Its objects correspond to the geometric
 deformation functors defined  by Manetti in \cite{Man2} (Corollary \ref{allequiv} and Remark \ref{manchar}). We then prove that our homotopy category  
is, in turn, is  equivalent to Kontsevich's category of SHLAs modulo tangent quasi-isomorphisms, as in \cite{Kon} (Proposition \ref{tqiscor}), and to the homotopy categories of DG coalgebras and DGLAs considered by Hinich in \cite{hinstack} (Corollary \ref{cfhin}). 
This shows that all existing approaches to derived deformations are essentially equivalent (Remarks \ref{cfothers} and Theorem \ref{mcallequiv}).

In Section \ref{sopsH}, we establish an Adams-type spectral sequence, enabling us to define a graded Lie algebra structure on the cohomology groups $\H^*(F)$ of any deformation functor. These are all the operations in characteristic $0$, but there are many additional operations in general, and we apply the model structure to outline the operations common to all deformation cohomologies.

I would like to thank the anonymous referees for their diligent work in identifying errors and omissions in the manuscript. I would also like to thank Andrey Lazarev for identifying an error in the published version of Lemma \ref{manrep}.

\section{Generalising smoothness}\label{one}

\subsection{Pro-categories}

In this section, we recall various background results.

\begin{definition}\label{proC}
Given a category $\C$, recall from \cite{descent} that the category  of pro-objects in $\C$, denoted $\pro(\C)$ or $\hat{\C}$, has objects consisting of filtered inverse systems $\{A_{\alpha}\in \C\}$, with 
$$
\Hom_{\pro(\C)}(\{A_{\alpha}\}, \{B_{\beta}\})= \lim_{\substack{\lla \\ \beta}} \lim_{\substack{\lra \\ \alpha}} \Hom_{\C}(A_{\alpha},B_{\beta}).
$$

The category $\ind(\C)$ of ind-objects is given by $\ind(\C)= \pro(\C^{\op})^{\op}$ (in other words, objects are filtered direct systems, and morphisms behave accordingly).
\end{definition}

\begin{definition}
A functor $F: \C \to \Set$ is said to be pro-representable (by $A$) if there exist $A \in \pro(\C)$ and a natural isomorphism
$$
F \cong \Hom_{\pro(\C)}(A, -)
$$
of functors from $\C$ to $\Set$.
\end{definition}

\begin{lemma}\label{profd}
The category $\pro(\FD\Vect_k)$ of pro-finite-dimensional vector spaces over a field $k$ is opposite to the category $\Vect_k$ of all vector spaces over $k$.
\end{lemma}
\begin{proof}
  There is a functor
$$
\varinjlim: \ind(\FD\Vect) \to \Vect
$$
from the  category of ind-finite-dimensional vector spaces to the category of all vector spaces, given by mapping a direct system $\{V_{\alpha}\}$ to $\varinjlim V_{\alpha}$. This is essentially surjective, since  any vector space is the  direct limit of its finite-dimensional subspaces. It is also full and faithful:
$$
\Hom_{\Vect}(\lim_{\substack{\lra\\ \alpha}} V_{\alpha}, \lim_{\substack{\lra\\ \beta}} W_{\beta})= \lim_{\substack{\lla\\ \alpha}}\Hom_{\Vect}(V_{\alpha}, \lim_{\substack{\lra\\ \beta}} W_{\beta})= \lim_{\substack{\lla\\ \alpha}}\lim_{\substack{\lra\\ \beta}}\Hom_{\Vect}(V_{\alpha},  W_{\beta}),
$$
since $V_{\alpha}$ is finite-dimensional. 

By taking duals, we see that $\ind(\FD\Vect)$ is equivalent to the opposite category of $\widehat{\FD\Vect}$.
\end{proof}

\begin{definition}\label{levelmap}
Recall from \cite{isaksenlim} that  if there  exists a cofiltered category $I$ and a system of morphisms $f_{\alpha}: X_{\alpha} \to Y_{\alpha}$ for $\alpha \in I$ in a category $\C$, then the resulting morphism $\{f_{\alpha}\}_{\alpha \in I}: \{X_{\alpha}\}_{\alpha \in I} \to \{Y_{\alpha}\}_{\alpha \in I}$ in $\pro(\C)$ is called a level map. By \cite{arma} Appendix 3.2, every morphism in $\pro(\C)$ is isomorphic to a level map. 
%
\end{definition}
%

\begin{definition}
We follow \cite{descent} in saying that an object in $\pro(\C)$ is strict if all the transition morphisms are epimorphisms.
\end{definition}

\begin{definition} As in \cite{descent},  we say that a functor is left exact if it preserves all finite limits. 
\end{definition}

\begin{definition}
Say that a pro-object $\{A_{\alpha}\}_{\alpha \in I}$  is saturated if it is strict, and for any $\alpha \in I$ and any epimorphism $A_{\alpha} \to B$, there exists a unique morphism  $\alpha \to \beta$ in $I$ such that $B \cong A_{\beta}$.  As observed in \cite{descent}, every strict pro-object is isomorphic to a saturated  pro-object. Beware that ``saturated'' is not standard terminology.
\end{definition}



\begin{lemma}\label{lexlemma}
For a functor $F:\C \to \Set$ on an Artinian category $\C$ with all finite limits, the following are equivalent
\begin{enumerate}
\item $F$ is left exact.
\item $F$ is pro-representable.
\item $F$ is pro-representable by a (saturated) strict pro-object.
\end{enumerate}
\end{lemma}
\begin{proof}
\cite{descent}, Corollary to Proposition 3.1.
\end{proof}

\subsection{Pro-Artinian algebras}

Fix a complete local Noetherian ring $\L$, with maximal ideal $\mu$ and residue field $k$. Let $\C_{\L}$ denote the category of local Artinian $\L$-algebras with residue field $k$. Let   $\hat{\C}_{\L}$ be its pro-category (as in Definition \ref{proC}).

\begin{remark}
Note that our definition of  $\hat{\C}_{\L}$   differs slightly from that in \cite{Sch} (which only admitted pro-Artinian rings with finite-dimensional cotangent spaces). Consequently our notion of pro-representability, which agrees with that in  \cite{descent}, is broader than that in  \cite{Sch}.
\end{remark}

Observe that epimorphisms in $\C_{\L}$ are precisely surjective maps.

\begin{definition}\label{smoothdef}
As in \cite{Man}, we say that a functor $F:\C_{\L}\to \Set$ is smooth if for all surjections $A \to B$ in $\C_{\L}$, the map $F(A) \to F(B)$ is surjective. 
\end{definition}

\begin{lemma}
There is a fully faithful embedding of  $\hat{\C}_{\L}$ into the category of Hausdorff topological  rings, denoted by $A \mapsto  \uleft{A}$.
\end{lemma}
\begin{proof}
Take $A \in \hat{\C}_{\L}$. By Lemma \ref{lexlemma}, we may assume that $A= \{A_s\}_{s \in S}$ is strict. Set  $\underleftarrow{A}: = \Lim_s A_s$; since $A$ is strict, the maps $\underleftarrow{A}\to A_s$ are surjective, so we may write $A_s= \underleftarrow{A}/I_s$. 

Define a topology on $\underleftarrow{A}$ by setting $\{a +I_s\,:\, a \in \underleftarrow{A},\, s \in S\}$ to be a basis of open neighbourhoods. Continuous morphisms are now precisely the morphisms in $\hat{\C}_{\L}$.
\end{proof}

\begin{remarks}
Note that giving a  strict pro-object $A= \{A_s\}_{s \in S}$ is equivalent to giving a  $\L$-algebra $ \underleftarrow{A}$ with a maximal ideal $\m(\uleft{A})$, together with a set $S$ of ideals contained in $ \m(\underleftarrow{A})$, with the properties that 
\begin{enumerate}
\item $\bigcap_{I  \in S} I=0$;
 \item  for all $I \in S$, the quotient $ \underleftarrow{A}/I$ is in $\C_{\L}$;
 \item If $I,J \in S$, then there exists $K \in S$ with $K \le I \cap J$ (weak closure).
\end{enumerate}

For a saturated pro-object, there is the additional condition that if $I \in S$ and $J \ge I$ is an ideal, then $J \in S$, and we may then replace weak closure with strong closure ($I, J \in S$ implies $I \cap J \in S$).

Observe that the saturated pro-object isomorphic to $A$ is $\{  \underleftarrow{A}/I\}_{I \in U}$, where $U$ is the set of all open ideals in $ \underleftarrow{A}$. 
\end{remarks}

\begin{definition}\label{surjdef}
Say a morphism $f: A \to B$ in $\hat{\C}_{\L}$ is surjective if the map $\uleft{f}: \uleft{A} \to \uleft{B}$ is surjective.
\end{definition}

\begin{remark}\label{surjchar}
If $\{  \underleftarrow{A}/I\}_{I \in S}$ is  saturated, then  subsets $T \subset S$ satisfying weak closure give rise to surjections with domain $A$, by setting $B= \{  \underleftarrow{A}/I\}_{I \in T}$. Every surjection with domain $A$ is isomorphic to one of this form, and we may also assume that if $I \in T$ and $I \le J \in S$, then $J \in T$ (which corresponds to $B$ being saturated).
\end{remark}

\subsection{Pro-Artinian simplicial algebras}

\begin{definition}\label{N^s}
Given a simplicial complex $V_{\bt}$ in an abelian category, recall that the normalised chain complex $N^s(V)_{\bt}$ is given by  $N^s(V)_n:=\bigcap_{i>0}\ker (\pd_i: V_n \to V_{n-1})$, with differential $\pd_0$. The simplicial Dold-Kan correspondence says that $N^s$ gives an equivalence of categories between simplicial complexes and non-negatively graded chain complexes in any abelian category. Where no ambiguity results, we will denote $N^s$ by $N$.
\end{definition}

\begin{lemma}\label{cotdef}
A simplicial complex $A_{\bt}$  of local $\L$-algebras with residue field $k$ and maximal ideal $\m(A)_{\bt}$   is Artinian if and only if:
\begin{enumerate}
\item the normalisation $N(\cot A)$ of the cotangent space $\cot A:=\m(A)/(\m(A)^2+\mu \m(A))$  is finite-dimensional (i.e. concentrated in finitely many degrees, and finite-dimensional in each degree). 
\item For some $n>0$, $\m(A)^n=0$.
\end{enumerate} 
\end{lemma}
\begin{proof}
This is just an adaptation of the standard proof for algebras.
The first condition is clearly necessary, since it is equivalent to saying that the simplicial vector space $\cot A$ is Artinian. The second condition is also necessary, since $\m(A)^n$ is a descending chain of simplicial ideals.
For sufficiency, use the standard filtration of $A$ by powers of $\m(A)$ and $\mu$, whose graded pieces are Artinian simplicial $k$-vector spaces. 
\end{proof}

\begin{definition} We define $s\C_{\L}$ to be the category of Artinian simplicial  local $\L$-algebras, with residue field $k$. 
\end{definition}

\begin{definition}\label{spdef}
Define $\Sp$, the category of spaces, to be  the category $(\hat{\C}_{\L})^{\op}$ (equivalent to the category  of left-exact functors from $\C_{\L}$ to $\Set$, since $\C_{\L}$ is Artinian). Given $R \in (\hat{\C}_{\L})^{\op}$, we let its formal spectrum $\Spf R$ be the corresponding object of the opposite category.
\end{definition}

\begin{proposition}\label{cSp}
The category $\pro(s\C_{\L})$ is equivalent to the category $s\hat{\C}_{\L}$ of simplicial objects in $\hat{\C}_{\L}$. 
\end{proposition}
\begin{proof}
There is a canonical functor $U:\pro(s\C_{\L})\to s\hat{\C}_{\L}$. Given $R \in s\hat{\C}_{\L}$, we may define a left-exact functor on $s\C_{\L}$ by $A \mapsto \Hom_{s\hat{\C}_{\L}}(R,A)$. Since $s\C_{\L}$ is Artinian, Lemma \ref{lexlemma} implies that  this is pro-represented by some $F(R)$ in $\pro(s\C_{\L})$. For $\{S(\alpha)\}_{\alpha}\in \pro(s\C_{\L})$, we then have
$$
\Hom_{\pro(s\C_{\L})}(F(R),\{S(\alpha)\} )= \Lim_{\alpha} \Hom_{s\hat{\C}_{\L}}(R, S(\alpha))= \Hom_{s\hat{\C}_{\L}}(R,U\{S(\alpha)\} ).
$$

Now, given $A \in \C_{\L}$, define $A^{\Delta_n}$  (not to be confused with $A^{\Delta^n}$) to be the simplicial ring 
$$
(A^{\Delta_n})_i:= \overbrace{A\by_k A \by_k \ldots \by_k A}^{\Delta^i_n},
$$
with $\pd_j: (A^{\Delta_n})_i \to (A^{\Delta_n})_{i-1}$ coming from  $\pd^j: \Delta^{i-1} \to \Delta^i$, and $\sigma_j$ coming from $\sigma^j: \Delta^{i+1} \to \Delta^i$. Clearly $ A^{\Delta_n} \in (\C_{\L})^{\Delta^{\op}}$, and since $N_iA^{\Delta_n}=0$ for all $i \ge n+2$, Lemma \ref{cotdef} implies that $A^{\Delta_n} \in s\C_{\L}$.

The key property of $A^{\Delta_n}$ is that for $R \in s\hat{\C}_{\L}$
$$
\Hom_{s\hat{\C}_{\L}}(R, A^{\Delta_n}) \cong \Hom_{\hat{\C}_{\L}}(R_n, A),
$$
which (taking colimits) implies that  for $S \in \pro(s\C_{\L})$, 
$$
\Hom_{\pro(s\C_{\L})}(S, A^{\Delta_n}) = \Hom_{\hat{\C}_{\L}}(S_n, A).
$$

Therefore
\begin{eqnarray*}
\Hom_{\hat{\C}_{\L}}((FR)_n, A) &\cong& \Hom_{\pro(s\C_{\L})}(FR, A^{\Delta_n})\\
&\cong& \Hom_{s\hat{\C}_{\L}}(R, A^{\Delta_n})\\
&\cong& \Hom_{\hat{\C}_{\L}}(R_n, A)
\end{eqnarray*}
for all $A \in \C_{\L}$, so $(FR)_n \cong R_n$, and the unit and counit of the adjunction $F\dashv U$ are both isomorphisms. This implies that the functors $F$ and $U$ are quasi-inverse.
\end{proof}

\begin{definition}\label{spfdef}
Define  $c\Sp:= \Sp^{\Delta}$, which is clearly opposite to the category $s\hat{\C}_{\L}$, and we denote this equivalence by $\Spf: (s\hat{\C}_{\L})^{\op} \to c\Sp$.
 Proposition \ref{cSp}  implies that $c\Sp$  is also equivalent to the category of left-exact functors from $s\C_{\L}$ to $\Set$.
\end{definition}

\begin{definition}
We say that a map $f:A \to B$ in $s\hat{\C}_{\L}$ is acyclic if $\pi_i(f):\pi_i(A) \to \pi_i(B)$ is an isomorphism of pro-Artinian $\L$-modules for all $i$.  $f$ is said to be surjective if each $f_n:A_n \to B_n$ is  surjective.
\end{definition}

Note that for any simplicial abelian group $A$, the homotopy groups can be calculated by $\pi_iA \cong \H_i(NA)$, the homology groups of the normalised chain complex. These in turn are isomorphic to the homology groups of the unnormalised chain complex associated to $A$. 

\begin{definition}
We define a small extension $e:I \to A \to B$ in $s\C_{\L}$ to consist of a surjection $A \to B$ in $s\C_{\L}$ with kernel $I$, such that $\m(A)\cdot I=0$. Note that this implies that $I$ is a simplicial complex of $k$-vector spaces.
\end{definition}

\begin{lemma}\label{small}
Every surjection in $s\C_{\L}$ can be factorised as a composition of small extensions. Every acyclic surjection in $s\C_{\L}$ can be factorised as a composition of acyclic small extensions.  
\end{lemma}
\begin{proof}
Let $f:A \to B$ be a surjection  in $s\C_{\L}$ with kernel $I$. Note that $N(A)$ has finite length, hence so does $NI$. We will prove the statements by induction on the length $l(NI)$. For $I=0$, both statements are trivial. 

If $I \ne 0$, then $l(N(\m(A) \cdot I)) < l( NI)$,  so the inductive hypothesis implies that $A \to A/\m(A)\cdot I$ can be factorised as a composition of small extensions. Since $A/\m(A)\cdot I \to B$ is a small extension, this gives a factorisation of $A \to B$ as a composition of small extensions.

If $f$ is acyclic, the argument takes more care. Let $V$ be a maximal acyclic quotient of $I/\m(A)\cdot I$, so that $d=0$ on $N(\ker(I/\m(A) \cdot I\to V))$. Let $J$ be the kernel of $I \to V$, so that
$
A/J \to B
$
is an acyclic small extension, having kernel $V$. 

Since $A \to A/J$ is also necessarily acyclic, the induction proceeds unless $J=I$, in which case $d=0$ on $N(I/\m(A)\cdot I)$. If so,  the long exact sequence of homology gives isomorphisms
$$
N_n(I/\m(A)\cdot I) \cong \left\{ \begin{matrix} \H_{n-1}N(\m(A)\cdot I) &n>0\\ 0 & n=0 \end{matrix}\right.
$$
Thus, if $n$ is the least such that $I_n \ne 0$, we have 
$$
I_n/(\m(A)\cdot I)_n = N_n(I/\m(A)\cdot I)=0,
$$ 
so $I_n=0$, giving the required contradiction. 
\end{proof}

\subsection{The model structure}

\begin{definition}
Denote the category of simplicial sets by $\bS$. 
\end{definition}

\begin{definition}\label{smcldef} In the category $c\Sp$, we say that $f:\Spf S\to \Spf R$ is:
\begin{enumerate}
\item
 a cofibration if the corresponding morphisms $ N_i(\uleft{f}^{\sharp}): N_i(\uleft{R}) \to N_i(\uleft{S})$ are surjective for all $i>0$ (cf. Definition \ref{surjdef});
\item
a weak equivalence if  $f^{\sharp}:R \to S$  is   acyclic;  
\item  
a fibration if it has the right lifting property (RLP) with respect to all trivial cofibrations.
\end{enumerate}
The simplicial structure is given by setting 
$$
(R\ten K)_i:= R_i^{\ten K_i},
$$
and 
$$
(R^K)_i=\Hom_{\bS}(K\by \Delta^i, R)\by_{\Hom_{\Set}(\pi_0K, k)}k,
$$
with $(\Spf R)^K=\Spf (R\ten K)$ and $(\Spf R)\ten K=\Spf (R^K)$.
\end{definition}

Observe that every surjection $A \onto B$ in $s\hat{\C}_{\L}$ is dual to a cofibration.

\begin{proposition}\label{smcl}
With the classes of morphisms  given above, $c\Sp$ is a simplicial model category.
\end{proposition}
\begin{proof}
We apply \cite{Bou} Theorem 12.4 and Proposition 3.13 to the category $\Sp$ with its discrete model structure.
By Lemma \ref{lexlemma}, every object in $s\hat{\C}_{\L}$ can be represented by a strict pro-object.

We therefore take the class $\cG$ of injective models to consist of the single functor $A \mapsto  \m(\uleft{A})$, i.e.
$$
\{A_s\}_{s \in S} \mapsto  \Lim_s \m(A_s).
$$ 
Thus a map $\Spf B \to \Spf A$ in $\Sp$ is $\cG$-monic when $A \to B$ is a surjection.  The class of $\cG$-injectives  therefore consists of   smooth  morphisms (in the sense of Definition \ref{smoothdef}) in $\Sp$.

For the model structure defined in \cite{Bou} 3.2,  
a map $f: X^{\bt} \to Y^{\bt}$, for  $X^{\bt} = \Spf B_{\bt}$, $Y^{\bt}=\Spf A_{\bt}$  is then:
\begin{enumerate}

\item a $\cG$-weak equivalence if $\uleft{f}^{\sharp}: \m(\uleft{A})_{\bt}  \to \m(\uleft{B})_{\bt} $ is a weak equivalence of simplicial groups;

\item a $\cG$-cofibration if $\m(A)_n \to \m(B)_n\by_{M_{\L^n_k}\m(B)_{\bt}}M_{\L^n_k}\m(A)_{\bt}$ is a surjection (in the sense of Definition \ref{surjdef}) for all $0\le k\le n$, where $\L^n_k\subset \Delta^n$ is the $k$th horn, and $M_KX:= \Hom_{\bS}(K,X)$; 
 
\item a $\cG$-fibration if the cosimplicial matching maps $X^n\to Y^n\by_{M^nY^{\bt}}M^nX^{\bt}$ are smooth for all $n \ge 0$. 
\end{enumerate}

From the Dold-Kan correspondence, we deduce that $f$ is a $\cG$-cofibration when for $i>0$, $N_i(\uleft{f}^{\sharp})$ is surjective.

Now observe that since every morphism in the category $\pro(\cM)$ of pro-Artinian $\L$-modules is isomorphic to a level map (as in Definition \ref{levelmap}), the functor
\begin{eqnarray*}
\Lim: \pro(\cM) &\to& \L-\Mod\\
\{M_{\alpha}\}_{\alpha \in I} &\mapsto& \Lim_{I} M_{\alpha} 
\end{eqnarray*}
 is  exact, so 
 $$
 \pi_i \m(\uleft{A})_{\bt}= \pi_i( \Lim \m(A)_{\bt})= \Lim \pi_i(A_{\bt}).
 $$
  
  In order to show  that $\cG$-weak equivalences are acyclic, it will suffice to prove that $\Lim$ reflects isomorphisms. Considering images under $\Lim$ of  kernels and  cokernels, we need only show that if  $\Lim_{I} M_{\alpha}\cong 0$, then $\{M_{\alpha}\}_{\alpha \in I}\cong 0$. By Lemma \ref{lexlemma}, every object in $\pro(\cM)$ is isomorphic to a strict pro-object, and $\Lim$ maps non-zero strict pro-modules to non-zero modules, as required.

To see that this model structure is simplicial, it is straightforward to verify \cite{sht} Proposition II.3.13. 
\end{proof}

\subsection{Properties of functors}

\begin{definition}
We say that a morphism $\alpha:F\to G$ in $c\Sp$  is smooth if for all small extensions $A \onto B$ in $s\C_{\L}$, the map $F(A) \to F(B)\by_{G(B)}G(A)$ is surjective. 

Similarly, we call $\alpha$ quasi-smooth if for all acyclic small extensions $A \to B$ in $s\C_{\L}$, the map $F(A) \to F(B)\by_{G(B)}G(A)$ is surjective.
\end{definition}

\begin{remarks}\label{aqsmoothchar}
 A quasi-smooth map $\alpha$ is smooth if
 the Andr\'e-Quillen homology groups $D_i(R/S)=0$ for all $i>0$, or equivalently the relative cotangent space $\cot(R/S)$ is acyclic in strictly positive degrees.

Our notion of quasi-smoothness will broadly correspond  to that used in \cite{Man2}. Some authors (e.g. \cite{toenseattle}) take quasi-smoothness to mean $D_i(R/S)=0$ for all $i>1$; this is a generalisation of LCI morphisms to simplicial rings, and is completely unrelated to our notion of quasi-smoothness.

However, our notion of smoothness differs from \cite{Man2} (where the term is only applied to functors on the homotopy category), and is stronger than that in \cite{hag2}. 
The latter roughly amounts to being smooth up to homotopy, or equivalently that the higher Andr\'e-Quillen homology groups vanish. Thus smoothness in our sense corresponds to quasi-smoothness (in our sense) plus formal smoothness in the sense of \cite{hag2}.

\end{remarks}

\begin{lemma}\label{qscofib}
A morphism $f:X^{\bt} \to Y^{\bt}$ in $c\Sp$ is a fibration if and only if it is quasi-smooth if and only if each $f^n:X^n \to Y^n$ is smooth.
\end{lemma}
\begin{proof}
By Lemma \ref{small}, we know that fibrations are precisely quasi-smooth maps.

If each $f^n$ is smooth, we may apply the Standard Smoothness Criterion (\cite{Man} Proposition 2.17) to deduce that the cosimplicial matching maps are smooth.

If $f$ is a cofibration, take a small extension $A \to B$ in $\C_{\L}$ and consider the acyclic small extension $A^{\Delta_n} \to B^{\Delta_n}$ in $s\C_{\L}$, for $A^{\Delta_n} $ as in the proof of Proposition \ref{cSp}. Observe that $X(A^{\Delta_n})= X^n(A)$, so quasi-smoothness of $f$ implies smoothness of $f^n$.
\end{proof}

\begin{definition}
Given a functor $F:\C_{\L} \to \Set$, we write $F:s\C_{\L} \to \Set$ to mean $A \mapsto F(A_0)$  (corresponding to the inclusion $\Sp \into c\Sp$).
\end{definition}

\begin{lemma}\label{ctod} A morphism  $\alpha:F\to G$ in $\Sp$ is smooth if and only if the induced morphism between the  objects $F, G \in c\Sp$ is quasi-smooth, if and only if it is smooth.
\begin{proof}
Apply Lemma \ref{qscofib}, noting that $F^0=F$. 
\end{proof}
\end{lemma}

\begin{definition}
Define the  $sc\Sp$ to be the category of left-exact functors from $s\C_{\L}$ to the category $\bS$ of simplicial sets.
\end{definition}

Now, observe that $sc\Sp$ is equivalent to the category $(c\Sp)^{\Delta^{\op}}$ of simplicial objects in $c\Sp$. We will make use of this identification without further comment.

We say that a morphism $X \xra{f} Y$ in $\bS$ is a surjective fibration if it is a fibration and $\pi_0(f)$ is surjective.
\begin{definition}\label{scspqsdef}
 A morphism $\alpha:F\to G$ in    $sc\Sp $ is  said to be smooth if 
\begin{enumerate}
\item[(S1)]
for every acyclic surjection $A \to B$ in $s\C_{\L}$, the map $F(A)\to F(B)\by_{G(B)}G(A)$ is a trivial fibration in $\bS$; 
\item[(S2)]
for every surjection $A \to B$ in $s\C_{\L}$, the map $F(A)\to F(B)\by_{G(B)}G(A)$ is a surjective fibration in $\bS$.
\end{enumerate}

A morphism $\alpha:F\to G$ in    $sc\Sp $  is  said to be quasi-smooth if it satisfies (S1) and
\begin{enumerate}
\item[(Q2)]
for every surjection $A \to B$ in $s\C_{\L}$, the map $F(A)\to F(B)\by_{G(B)}G(A)$ is a  fibration in $\bS$.
\end{enumerate}
\end{definition}

\begin{remark}
In \cite{ddt2} \S \ref{ddt2-extsdc}, it is shown that a quasi-smooth object of $sc\Sp$ can be canonically associated to all deformation problems governed by the SDCs of \cite{paper1} and \cite{paper2}. This includes all classical deformation problems, such as deformations of an arbitrary scheme.
\end{remark}

\begin{definition}\label{underline}
Given   $F \in sc\Sp$, define $\underline{F}:s\C_{\L}\to \bS$ by 
$$
\underline{F}(A)_n:= F_n(A^{\Delta^n}).
$$
Observe that if $F=\Hom(R,-):s\C_{\L} \to \Set$, for $R \in s\hat{\C}_{\L}$,  then $\underline{F}=\HHom(R,-)$.
 
For $F \in c\Sp$, we may regard $F$ as an object of $sc\Sp$ (with the constant simplicial structure), and then define $\underline{F}$  as above.
\end{definition}

\begin{lemma}\label{settotop} A map $\alpha:F\to G $ in $c\Sp$ is smooth (resp. quasi-smooth) if and only if the induced map of functors  $\underline{\alpha}:\underline{F}\to \underline{G}$  is smooth (resp. quasi-smooth) in $sc\Sp$.
\begin{proof}
 This follows from the fact that $s\C_{\L}$  is a simplicial model category, and that every surjection is a fibration. If we pro-represent $\alpha$ by $R \to S$ in $s\hat{\C}_{\L}$, then quasi-smoothness of $\underline{\alpha}$ is equivalent to the conditions:
\begin{enumerate} 
\item for all cofibrations $K \into L$ in $\bS$, $\theta:(R\ten L)\ten_{R\ten K}(S\ten K) \to S\ten L$ is quasi-smooth;

\item if in addition $K \into L$ is a weak equivalence, then $\theta$ is smooth.
\end{enumerate}
Smoothness of $\underline{\alpha}$ is then just the further condition that $\alpha$ be smooth.
\end{proof}
\end{lemma}

\begin{definition}\label{pioqs}
A map $ \alpha:F\to G$ of functors   $F,G:\C_{\L}\to \bS$ is said to be smooth (resp. quasi-smooth, resp. trivially smooth)  if for all surjections $A \onto B$ in $\C_{\L}$, the maps 
$$
 F(A) \to F(B)\by_{G(B)}G(A)
$$ 
  are surjective fibrations (resp. fibrations, resp. trivial fibrations).
\end{definition}

\begin{definition}\label{tdef0}
Given a left-exact functor $F: \C_{\L} \to \Set$, define the tangent space $t_F$ (or $t(F)$) by $t_F:= F(k[\eps]/(\eps^2))$. Since  $k[\eps]/(\eps^2)$ is an abelian group object in $\C_{\L}$, $t_F$ is an abelian group. The endomorphisms $\eps \mapsto \lambda \eps$ of $k[\eps]/(\eps^2)$ make $t_F$ into a  vector space over $k$. 

Given a  left-exact functor $F: \C_{\L} \to \bS$, define the simplicial vector space $t_F$ by $(t_F)_n:= t(F_n)$. 
\end{definition}

\begin{proposition}\label{smoothchar}
A map  $ \alpha:F\to G$  of left-exact functors   $F,G:\C_{\L}\to \bS$ is smooth if and only if the maps  $F_n\xra{\alpha_n} G_n$ of functors   $F_n,G_n:\C_{\L}\to \Set$ are all smooth.
\end{proposition}
\begin{proof}
If $X \to Y$ is a surjective fibration in $\bS$, then it follows from the right lifting property for fibrations that the maps $X_n \to Y_n$ are surjective. Therefore, if  $F\xra{\alpha} G$ is smooth, the maps $F_n\xra{\alpha_n} G_n$ are all smooth.

Conversely, assume that $\alpha_n$ is smooth for all $n$. Since every surjection in $\C_{\L}$ is a composition of small extensions, it suffices to show that for every small extension $A \onto B$ in $\C_{\L}$, with kernel $I$, the map $F(A) \xra{\beta} F(B)\by_{G(B)}G(A)$ is a surjective fibration. Now, by left-exactness,
$$
F(A)\by_{F(B)}F(A)=F(A\by_BA)\cong F(A \by (k\oplus I\eps))=F(A) \by t_F\ten I,
$$ 
where $\eps^2=0$, so $F(A)$ has a faithful action by the additive group $t_F \ten I$, the quotient being isomorphic to the image of $F(A) \to F(B)$. The same formulae hold for $G$, and if we let $H=\ker(t_F \ten I \to t_G\ten I)$, we see that $F(A)/H$ is isomorphic to $F(B)\by_{G(B)}G(A)$, since $F(A)$ maps onto this, by hypothesis. Therefore, by \cite{sht} Corollary V.2.7, $\beta$ is a surjective fibration, so $\alpha$ is smooth, as required.
\end{proof}

\begin{proposition}\label{ctodt1} If a morphism $\alpha:F\to G $ in $sc\Sp$  is such that the map
$$
\theta: F(A)\to F(B)\by_{G(B)}G(A)
$$
is a surjective fibration  for all acyclic small extensions $A \to B$, then $\underline{\alpha}:\underline{F} \to \underline{G}$ satisfies condition (S1) from Definition \ref{scspqsdef}. 
\end{proposition}
\begin{proof}
Given $A \in s\C_{\L}$, consider the  bisimplicial sets $ F(A^{\Delta^{\bt}}),  G(A^{\Delta^{\bt}})$. We wish to show that 
$$ 
\theta:F(A^{\Delta^{\bt}})\to   G(A^{\Delta^{\bt}})\by_{G(B^{\Delta^{\bt}})} F(B^{\Delta^{\bt}})
$$ 
is a  diagonal trivial fibration for all acyclic small extensions $A \to B$.

Now, if $A \to B$ is an acyclic  small extension, then $A^L \to B^L\by_{B^K}A^K$ is also an acyclic small extension for all cofibrations $K \to L$ in $\bS$.
Thus $\theta$ is a Reedy fibration, since $\Hom_{\bS}(K, \underline{F_n}(A))= F_n(A^K)$ for all finite simplicial sets $K$, by left-exactness. Moreover, for fixed $m$, $\alpha_m:F_m\to G_m$ is quasi-smooth, for  $F_m,G_m: s\C_{\L} \to \Set$. By Lemma \ref{settotop}, this implies that $\underline{\alpha_m}$ is quasi-smooth, so $\theta_m$ is a Kan fibration. Thus $\theta$ is a Reedy fibration and a horizontal Kan fibration, so \cite{sht} Lemma IV.4.8 implies that $\diag \theta$ is a fibration. 

Finally, the quasi-smoothness of $\underline{\alpha_m}$ implies that $\theta_m$ is a weak equivalence for all $m$. \cite{sht}  Proposition IV.1.7 then implies that $\diag\theta$ is a weak equivalence.
\end{proof}

\begin{example}\label{BG}
If $G:\C_{\L} \to \Gp$ is a smooth left-exact group-valued functor, then the classifying space $B\underline{G}: s\C_{\L} \to \bS$ is smooth, but not a right  Quillen functor for the simplicial model structure.
\end{example}

\subsection{Cohomology and obstructions}\label{cohomology}

\begin{definition}\label{weakdef}
We will say that a morphism $\alpha: F \to G$ of quasi-smooth objects of $sc\Sp$ is a weak equivalence if, for all $A \in s\C_{\L}$, the maps $\pi_i F(A) \to \pi_iG(A)$ are isomorphisms for all $i$.
\end{definition}

\begin{definition}\label{tgtdef}
Given    $F \in sc\Sp$, define the  tangent space functor $\tan F:s\FD\Vect_k \to \bS$, on simplicial  $k$-vector spaces with finite-dimensional normalisation, by $\tan F(V):= F(k\oplus V)$, where the multiplication is given by $V^2=0$. Given a morphism $\alpha: F \to G$ of left-exact functors, define $\tan(F/G):= \ker(\tan F \to \tan G)$.
 
Similarly to  Definition \ref{tdef0}, $(\tan F)(V)$ has a natural vector space structure inherited from $V$. Thus we regard $\tan F$ as a functor
$$
\tan F:s\FD\Vect_k \to s\Vect_k.
$$
\end{definition}

\begin{definition}\label{knlndef}
Let $K^n:=N^{-1}k[-n] \in s\FD\Vect_k$, and $L^n:=N^{-1}(k[-(n+1)]\xra{\id} k[-n])$, noting that $\pi_*K^n= k[-n]$ and $\pi_*L^n=0$. For $V \in s\FD\Vect_k$, let $V[-n]:=V\ten K^n$. 
\end{definition}

\subsubsection{Obstruction maps}
We have the following characterisation of obstruction theory:

\begin{theorem}\label{robs}
If  $\alpha:F \to G$ in $sc\Sp$ is quasi-smooth, then for any small extension $e:I \to A \xra{f} B$ in $s\C_{\L}$, there is a sequence of sets
$$
\pi_0(FA)\xra{f_*} \pi_0(FB\by_{GB}GA) \xra{o_e} \pi_0\tan(F/G)(I[-1]) 
$$  
exact in the sense that the fibre of $o_e$ over $0$ is the image of $f_*$. Moreover,  there is a group action of $\pi_1\tan(F/G)(I[-1])$ on $\pi_0(FA)$ whose orbits are precisely the fibres of $f_*$. 

For any $y \in F_0A$, with $x=f_*y$, the fibre of $FA \to FB\by_{GB}GA$ over $x$ is isomorphic to $\ker(\alpha: \tan F (I) \to \tan G(I))$, and the sequence above 
extends to a long exact sequence
$$\xymatrix@R=0ex{
\cdots  \ar[r]^-{f_*}&\pi_n(FB\by_{GB}GA,x) \ar[r]^-{o_e}& \pi_{n}\tan (F/G) (I[-1]) \ar[r]^-{\pd_e} &\pi_{n-1}(FA,y)\ar[r]^-{f_*}&\cdots\\ 
\cdots \ar[r]^-{f_*}&\pi_1(FB\by_{GB}GA,x) \ar[r]^-{o_e}& \pi_1\tan (F/G) (I[-1])  \ar[r]^-{-*y} &\pi_0(FA).
}
$$
\end{theorem}
\begin{proof}
Let $C=C(A,I):= (A \oplus I\ten L^0\eps)/(e +\eps)I$ be the mapping cone of $e$, where $\eps^2=0$. Then $(f,0):C \to B$ is a small acyclic surjection, so $FC\by_{GC}GA \to FB\by_{GB}GA$ is a weak equivalence, and thus $\pi_i(FC\by_{GC}GA)  \to \pi_i(FB\by_{GB}GA)$ is an isomorphism for all $i$.

Now,
$
A = C\by_{k \oplus I[-1]\eps}k,
$ 
and since $C\to k \oplus I[-1]\eps$ is surjective, this gives  a fibration
$$
p':FC \to \tan F(I[-1])\by_{\tan G(I[-1])}GC,
$$
which pulls back along $GA \to GC$ to give a fibration
$$
p:FC\by_{GC}GA \to \tan (F/G)(I[-1]),
$$
with fibre $FA$ over $0$.

The result now follows from the long exact sequence of homotopy (\cite{sht} Lemma I.7.3) for the fibration $p$, with the obstruction maps given by $p_*$.   
\end{proof}

\begin{corollary}\label{cohowelldfn}
For $F,G$ as above, there are canonical isomorphisms  $\pd:\pi_{i+1}\tan(F/G)(V[-n-1])\xra{\cong} \pi_i\tan(F/G)(V[n])$ for all $i$ and $V \in s\Vect_k$. Under this isomorphism, the boundary map $\pd_e$ in Theorem \ref{robs} is given by $\pd_e=e_*\circ \pd$, for $e:F(I) \to F(A) $ the fibre over $x$. 
\end{corollary}
\begin{proof}
The first statement follows from considering  the small extension $V\ten K^n\to V \ten L^n \to V\ten K^{n+1}$. For the second, 
the isomorphism $A\by_B A \cong A \by (k \oplus I\eps)$ gives an isomorphism $F(A)\by_{F(B)}F(A) \cong F(A) \by F(I)$, the result following by functoriality.
\end{proof}

\begin{definition}\label{cohodef}
For $\alpha:F \to G$ as above, we define  $\H^{j}(F/G)= \pi_i\tan (F/G)(K^n)$ for any $n-i=j$.
Given $V_{\bt} \in s\FD\Vect$, define $H^i(F/G\ten V):=\bigoplus_{n \ge 0} \H^{i+n}(F/G) \ten \pi_n(V)$, for $i \in \Z$. 

If $G=\bt$ (the one-point set), we write $\H^j(F):= \H^j(F/\bt)$.
\end{definition}

\begin{remarks}
This means that we may replace $\pi_{n}\tan (F/G) (I[-1])$ by $\H^{1-n}(F/G \ten I)$ in Theorem \ref{robs}.
To understand how this relates to classical obstruction theories, note that classical deformation functors are of the form $\pi_0F$, with $\pi_1F$ being (outer) automorphisms, and the $\pi_nF(A)$ corresponding to higher homotopies, which vanish for most classical problems when $A \in \C_{\L}$. We are   accustomed to tangent and obstruction spaces arising as $\H^1$ and $\H^2$ of a cohomology theory, rather than $\H^0$ and $\H^1$; essentially this is because $\bS$ models  classifying spaces of (simplicial) groupoids (similarly to Example \ref{BG}), and $\pi_i\bar{W}G=\pi_{i-1}G$.
\end{remarks}

\begin{corollary}\label{weak} 
A map $\alpha:F \to G$ of quasi-smooth  $F,G\in sc\Sp$ is a weak equivalence if and only if 
the maps $\H^j(\alpha):\H^j(F) \to\H^j(G)$ are all isomorphisms.
\end{corollary}

\begin{corollary}\label{cohosmoothchar}
If $\alpha:F \to G$ is  quasi-smooth in $sc\Sp$, then $\alpha$ is smooth if and only if $\H^i(F/G)=0$ for all $i>0$.
\end{corollary}
\begin{proof}
If $\alpha$ is smooth, then $\tan F(L^n) \to \tan F(K^{n+1})\by_{\tan G(K^{n+1})}\tan G(L^n)$ is surjective for all $n$. Since $\pi_0 \tan (F/G)(L^n)=0$, the long exact sequence of homology then gives $\pi_0\tan (F/G)(K^{n+1})=0$, so $H^i(F/G)=0$ for all $i>0$. The converse follows from Theorem \ref{robs}.
\end{proof}

\subsubsection{Properties of cohomology}\label{relativecoho}

\begin{definition}
Let $\Spf: ((s\hat{\C}_{\L})^{\Delta})^{\op} \to sc\Sp$ be the equivalence given by extending Definition \ref{spfdef} to simplicial diagrams. 
\end{definition}

\begin{definition}
Given a  map  $\alpha:F \to G$ between  $F, G \in sc\Sp$, for $F=\Spf S, G=\Spf R$, define the cotangent space by
$$
\cot(S/R):= \m(S)/(\m(S)^2 + S\cdot\m(R) ):\bS \to s\widehat{\FD\Vect}.
$$
\end{definition}

\begin{definition}\label{qsvect} We say that a left-exact functor $T: s\FD\Vect_k \to \bS$  is quasi-smooth if   it maps  acyclic surjections to trivial fibrations and surjections to fibrations.
\end{definition}

Standard properties of simplicial complexes then give:
\begin{lemma}
If   $\alpha$ as above  is quasi-smooth, then $\cot (S/R): \bS \to s\widehat{\FD\Vect}_k$ is quasi-smooth, in the sense that its left adjoint is so.
\end{lemma}

Under the cosimplicial Dold-Kan correspondence, the category of cosimplicial complexes over an abelian category is equivalent to the category of (non-negatively graded) cochain complexes over that category. This correspondence sends $F$ to its conormalisation $(N_cV(\Delta^{\bt}))^n= V(\Delta^n)/V(\L^n)$, where $\L^n$ denotes the $0$th horn of $\Delta^n$ (or $\emptyset$ if $n=0$), the differential being $ d=\sum_i (-1)^i \pd^i$.

\begin{definition}
Given a cosimplicial simplicial complex $V^{\bt}_{\bt}$, define the cochain complex of chain complexes
$$
N V^{\bt}_{\bt}:= N^s N_c V^{\bt}_{\bt}
$$
by making double use of the Dold-Kan correspondence, combining  cosimplicial conormalisation with the simplicial normalisation of Definition \ref{N^s}.   Write $d^s$ for the chain differential, and $d_c$ for the cochain differential. 
\end{definition}

\begin{lemma}\label{cohochar}
If  $\alpha:F \to G$  is quasi-smooth in $sc\Sp$,  then for $n>0$, $\H^n(F/G)$ is dual to $\H_n^{d^s}(N\cot (S^0/R^0))$. For $n\le 0$,  $\H^n(F/G)$ is dual to $\H^{-n}_{d_c}(\H_0^{d^s}(N\cot (S/R)^{\bullet}))$. 
\end{lemma}
\begin{proof}
Write $V:=\cot (S/R)$, and so $V(\Delta^{\bt}):= \cot (S(\Delta^{\bt})/R(\Delta^{\bt}))$.

The first condition of quasi-smoothness is that $V(\pd\Delta^n) \to V(\Delta^n)$ is injective for all $n$; this is  equivalent to saying that $\H^n(  N_c V(\Delta^{\bt})  )=0 \in s\widehat{\FD\Vect}_k$ for all $n$. The second condition is that $V(\L^n) \to V(\Delta^n)$ is quasi-trivial in $s\FD\Vect_k$ for $n>0$; this is  equivalent to saying that $\pi_i(N_cV(\Delta^{\bt}) )^n=0$ for all $i>0$ and $n>0$.

We may use the Dold-Kan equivalence again, and consider $NV(\Delta^{\bt}):=N^s N_c V(\Delta^{\bt})$, which is a cochain complex of chain complexes. Now, the simplicial complex $\tan (F/G)(K^n)$ is given by
$$
\tan (F/G) (K^n)_i= \Hom_{dg\widehat{\FD\Vect}_k}(N_sV(\Delta^i), k[-n]),
$$
where $dg\widehat{\FD\Vect}_k$ is the category of pro-finite-dimensional non-negatively graded chain complexes over $k$.
Thus the chain complex $N_s\tan (F/G)(K^n)$ is dual to the cochain complex $(N V(\Delta^{\bt})_n)/(d^sNV(\Delta^{\bt})_{n+1})$, where $d^s$ denotes the chain differential. 

If we write 
\begin{eqnarray*}
Z_n^{\bullet}&:=&\ker(d^s:  NV(\Delta^{\bt})_n \to NV(\Delta^{\bt})_{n-1})\\
B_n^{\bullet}&:=&\im(d^s:  NV(\Delta^{\bt})_{n+1} \to NV(\Delta^{\bt})_{n})\\
H_n^{\bullet}&:=&Z_n^{\bullet}/B_n^{\bullet},
\end{eqnarray*}
there is then a short exact sequence $0 \to H_n^{\bullet}\to (NV(\Delta^{\bt})_n)/B_n^{\bullet} \xra{d_s} B_{n-1}^{\bullet} \to 0$. The first condition of quasi-smoothness implies that $NV(\Delta^{\bt})_{n-1}$ is acyclic, while the second implies that $H_n^{\bullet}$ is concentrated in degree zero for $n>0$. From the former, we deduce that $\H^0(B_{n-1}^{\bullet})=0$, the latter then giving an isomorphism $\H^0((NV(\Delta^{\bt})_n)/B_n^{\bullet})\cong (H_n^{\bullet})^0$, as required.
\end{proof}

\begin{definition}\label{ntanfdef}
Define $t(F/G)$ to be the dual of $\cot(S/R)$; this is a cosimplicial complex of simplicial complexes over $k$, by Lemma \ref{profd}. Let $N_ct(F/G)$ be the cosimplicial normalisation of $t(F/G)$, and observe that this is a cochain complex of simplicial complexes, dual to $N^s\cot(S/R)$. Let $Nt(F/G):= N^sN_ct(F/G)$, the binormalised tangent complex. This is dual to $N_cN^s\cot(S/R)$.

Let $t(F):=t(F/\bt)$, and define the total complex 
$$
(\Tot N t(F))^n:=\bigoplus_{a-b=n} (N\tan t(F))^a_b,
$$
with coboundary operator given by $d_c\pm d^s$.
\end{definition}

\begin{lemma}
$t(F/G)$ is related to $\tan(F/G)$ by the formula
$$
(N_ct(F/G))^n= \tan(F/G)(L^n),
$$
for $L^n$ as in Definition \ref{knlndef}.
\end{lemma}
\begin{proof}
This is just the observation that for any $V\in s\widehat{\FD\Vect}$, 
$$
\Hom_{s\widehat{\FD\Vect}}(V, L^n) \cong \Hom_{\widehat{\FD\Vect}}( N_s(V)_n, k),
$$
applied to $V= \cot (S/R)^i$ for all $i$.
\end{proof}

\begin{proposition}\label{totcoho}
There are natural isomorphisms of cohomology groups 
$$
\H^n(F/G)\cong \H^n(\Tot N t(F/G)).
$$ 
\end{proposition}
\begin{proof}
Consider the spectral sequence 
$$
E_2^{a,-b}=\H_b(\H^a(Nt (F/G)))\implies \H^{a-b}(\Tot Nt (F/G)).
$$
This spectral sequence converges (coming from  a fourth quadrant double complex in the terminology of \cite{W} p.142). If we set
$$
W^n:= \left\{\begin{matrix} (Nt (F/G))^n_0 & n \ge 0 \\  Z^0_{d_c}(Nt (F/G))_{-n} & n < 0, \end{matrix} \right.
$$
then the map $ W^{\bullet} \to(\Tot N t (F/G))^{\bullet}$ gives an isomorphism on spectral sequences, and hence on cohomology (since both spectral sequences are strongly convergent). Finally, Lemma \ref{cohochar} implies that the cohomology of $W$ is just the cohomology of $(F/G)$.
\end{proof}

The following is immediate.
\begin{lemma}
If $X,Y,Z:s\C_{\L} \to \bS$ are left-exact, and   $X \xra{\alpha} Y$ is a quasi-smooth map, with $\beta:Z \to Y$ any map,  set $T:=X\by_YZ $. Then $T \to Z$ is quasi-smooth, and  there is  an isomorphism
$$
\H^*(T/Z) \cong \H^*(X/Y).
$$ 
\end{lemma}

\begin{proposition}\label{longexact}
Let $X,Y, Z:s\C_{\L}\to \bS$ be left-exact functors, with  $X \xra{\alpha} Y$ and $Y \xra{\beta} Z$ quasi-smooth.  There is then a long exact sequence
$$
\ldots \xra{\pd} \H^j(X/Y) \to \H^j(X/Z) \to \H^j(Y/Z) \xra{\pd} \H^{j+1}(X/Y) \to \H^{j+1}(X/Z) \to \ldots
$$
\end{proposition}
\begin{proof}
Since $t(X/Y)=\ker(\alpha:t (X) \to t (Y))$, we have a short exact sequence of bicomplexes
$$
0 \to Nt(X/Y) \to Nt(X/Z) \to Nt(Y/Z) \to 0,
$$
giving the required long exact sequence.
\end{proof}

\begin{lemma}\label{underlinecoho}
For a map $F\xra{\alpha} G$ of left-exact functors $F,G:s\C_{\L} \to \bS$, the relative tangent space $t(\underline{F}/\underline{G} )$ is given by the simplicial cosimplicial complex
$$
 t (\underline{F}/\underline{G})_n^i= (t(F/G)^i_n)^{\Delta^n_i}.
$$
In particular, $\H^i(t (\underline{F}/\underline{G}))\cong \H^i(t(F/G)) \in s\Vect_k$ for all $i$.
\end{lemma}
\begin{proof}
If  $F$ is represented by $R$, then $(\underline{F})_n$ is represented by $R^n\ten \Delta^n$; the description of $t (\underline{F}/\underline{G})$ follows immediately. 

For all $n$ and all cosimplicial complexes $V$, the morphisms $V^i \to (V^i)^{\Delta^n_i}$ fit together to form a cosimplicial quasi-isomorphism, dual to the quasi-isomorphism $V^{\vee}\ten \Delta^n \to V^{\vee}$ of simplicial complexes. This gives the isomorphism on cohomology.  
\end{proof}

\begin{proposition}\label{ctodtnew}
If a morphism $F\xra{\alpha} G$ of left-exact functors $F,G:s\C_{\L} \to \bS$ is such that the maps
$$
\theta: F(A)\to F(B)\by_{G(B)}G(A)
$$
 are surjective fibrations  for all acyclic small extensions $A \to B$, then $\underline{\alpha}:\underline{F} \to \underline{G}$ is quasi-smooth (resp. smooth) if and only if the groups $\H^i(t(F/G))$ are constant simplicial complexes (resp. $0$) for all $i>0$. This is equivalent to saying that $\theta$ is a fibration (resp. surjective fibration) for all small extensions $A \to B$.
\end{proposition}
\begin{proof}
By Proposition \ref{ctodt1}, we know that $\underline{\alpha}$ satisfies (S1), so we only need to prove (Q2) (resp. (S2)). Given a simplicial set $K$, write $M_KX:= \Hom_{\bS}(K,X)$, for $X \in \bS$. For any trivial cofibration $K \into L$ between simplicial sets (resp. any such cofibration and $\emptyset \to \bt$), we must demonstrate that 
$$
M_L\underline{F} \to M_K\underline{F}\by_{M_K\underline{G}}M_L\underline{G}
$$
is smooth. By Lemma \ref{smoothchar} and Corollary \ref{cohosmoothchar}, this is equivalent to showing that 
$$
\H^i(t(M_L\underline{F}/ M_K\underline{F}\by_{M_K\underline{G}}M_L\underline{G}))=0
$$
for all $i>0$.

This, in turn, is equivalent to saying that $d:t(\underline{F}/\underline{G})^{i-1} \to \z^i(t(\underline{F}/\underline{G}))$ is a fibration (resp. surjective fibration) for all $i>0$, or that $\H^i(t(\underline{F}/\underline{G}))$ is a constant simplicial complex (resp. $0$). By Lemma \ref{underlinecoho}, this is equivalent to asking that $\H^i(t(F/G))$ be a constant simplicial complex (resp. $0$), as required.

Finally, observe that these conditions are equivalent to asking that 
$$
\H^i(t(M_LF/ M_KF\by_{M_KG}M_LG))=0
$$
for all such $K \into L$, which is the same as saying that $\theta$ is a fibration (resp. surjective fibration) for all small extensions $A \to B$.
\end{proof}

\section{Model structures}\label{model}

\subsection{Cosimplicial spaces}

\begin{definition}
Define $I_{\Sp}$ to be the class of morphisms $f:X \to Y$ in $c\Sp$ for which either $f$ is dual to a small extension in $s\C_{\L}$, or both $X,Y \in \Sp$. Define  $J_{\Sp}$ to consist of those $f$ dual to acyclic small extensions in $s\C_{\L}$.
\end{definition}

\begin{remark}
Observe that the set of isomorphism classes in $\C_{\L}$ is small (since all local Artinian rings are quotients of finitely generated polynomial rings). We may therefore replace $I_{\Sp}, J_{\Sp}$ by small subsets, justifying the use of the small object argument which follows.
\end{remark}

\begin{lemma}
The model category $c\Sp$ is cofibrantly generated, with $I_{\Sp}$ the generating cofibrations, and $J_{\Sp}$ the generating trivial cofibrations.
\end{lemma}
\begin{proof}
First note that elements of $I_{\Sp}$ are clearly cofibrations, and similarly for $J_{\Sp}$. Given a fibration $R \to S$ in $s\hat{\C}_{\L}$, note that $\pi_0R \to \pi_0S$ is in $I_{\Sp}$, so $S\by_{\pi_0S}\pi_0R \to S$ is in the class $I_{\Sp}$-cell, and that $R\to  S\by_{\pi_0S}\pi_0R$ is surjective.  Lemma \ref{small} now implies that $R \to S$ is in $I_{\Sp}$-cell. Likewise, Lemma \ref{small} implies that acyclic surjections are precisely $J_{\Sp}$-cell complexes.
\end{proof}

\subsection{Simplicial cosimplicial spaces}\label{defcat}

\begin{definition}
Given $X \in \Sp$, with $X=\Spf R$, write $O(X):=R \in \hat{\C}_{\L}$.
\end{definition}

\begin{definition}
Given $X\in sc\Sp$, and $K \in \bS$, define $X\otimes K \in sc\Sp$ by  
$$
O(X\otimes K)_n^i:= \overbrace{O(X)^i_n\by_k  O(X)^i_n \by_k \ldots\by_k  O(X)^i_n }^{K_i}.
$$

Given $X \in sc\Sp$, $K \in \bS$, we define $X^K$ by $X^K(A):= (X(A))^K$, for $A \in s\C_{\L}$.
\end{definition}

\begin{definition}\label{Xpdef}
Given a quasi-smooth map $E \xra{p} B$ in $sc\Sp$, and a morphism $X\to B$  in $sc\Sp$, define $[X,E]_B$ to be the coequaliser
$$
\xymatrix@1{\Hom_{sc\Sp\da B}(X,E^{\Delta^1}\by_{B^{\Delta^1}}B) \ar@<.5ex>[r]\ar@<-.5ex>[r] &\Hom_{sc\Sp\da B}(X,E) \ar[r] &[X,E]_B},
$$ 
which can be thought of as maps from $X$ to $E$ over $B$, modulo fibrewise homotopy equivalences over $B$. 
\end{definition}
Note that for $f:X \to B$, we have $[X,E]_B \cong [X,f^*E]_X$, where $f^*E= E\by_BX$.

\begin{definition}\label{defcatdef} Given  a map $f:X \to Y$ in the category $sc\Sp$, say that $f$ is:  
\begin{enumerate}
\item a geometric cofibration if the corresponding morphism $(f^{\sharp})^n_i: O(Y)^n_i \to O(X)^n_i$ is surjective for all $i,n\ge 0$;

\item a geometric weak equivalence if for all quasi-smooth maps $p:E\to Y$, 
$$
f^*:[Y,E]_Y \to [X,E]_Y
$$
is an isomorphism;

\item  
a geometric fibration if $f$ is quasi-smooth.
\end{enumerate}
\end{definition}

\begin{definition}\label{abuse}
Given categories $\C, \cD$ and classes $P, Q$ of morphisms in $\C, \cD$ respectively, we will systematically abuse terminology by saying that a natural transformation $F \to G$ of functors $F,G: \C \to \cD$ ``maps $P$ to $Q$'' if for all morphisms $f:A \to B$ in $P$, the morphism
$$
F(A) \to G(A)\by_{G(B)}F(B)
$$
is in $Q$. Note that when $G$ is the constant functor to the final object of $\cD$, this amounts to saying that $F$ maps the class $P$ to the class $Q$.
\end{definition}

\begin{lemma}\label{snexists}
If $f:X \to Y$ is quasi-smooth in $sc\Sp$, with the map 
$$
\theta:X(A) \to X(B)\by_{Y(B)}Y(A)
$$
 a weak equivalence in $\bS$ for all small extensions $A \to B$ in $s\C_{\L}$, then $f$ has a section in $sc\Sp$.
\end{lemma}
\begin{proof}
The conditions state that $X\to Y$ maps small extensions in $s\C_{\L}$ to trivial fibrations in $\bS$, or equivalently that the simplicial matching maps
$$
X_n \to Y_n \by_{M_n Y} M_nX
$$ 
are trivial fibrations in $c\Sp$ for all $n$.

We construct the section inductively on $n$. Assume that there are compatible sections $Y_i \to X_i$ for all $i < n$. In particular, this gives $M_nY \to M_nX$. Now consider the commutative diagram
$$
\xymatrix{
L_nY \ar[r] \ar[d] & X_n \ar[d]\\
Y_n \ar[r] \ar@{-->}[ur] &  Y_n\by_{M_n Y}M_nX,}
$$
in $c\Sp$; the left-hand side is a cofibration, and the right-hand side a trivial fibration, so the dashed arrow exists.
\end{proof}

\begin{lemma}\label{smallworks}
A quasi-smooth map $f:X \to Y$ is a geometric weak equivalence in $sc\Sp$ if and only if for all small extensions $A \to B$ in $s\C_{\L}$, the map
$$
\theta:X(A) \to X(B)\by_{Y(B)}Y(A)
$$
is a weak equivalence in $\bS$.
\end{lemma}
\begin{proof} 
If $f:X \to Y$ is a geometric weak equivalence, then  $f^*:[Y,X]_Y \to [X,X]_Y$ must be an isomorphism.
Thus the identity map $\id: X \to X$ in $[X,X]_Y$ must lift to $[Y,X]_Y$, giving a section $s:Y \to X$ of $f$, and a homotopy $h:X \to X^{\Delta^1}\by_{Y^{\Delta^1}}Y$ between $\id$ and $sf$. For all small extensions $A \to B$, these data make the fibration $\theta$ into a deformation retract, and hence a weak equivalence.

Conversely, if $\theta$ is a weak equivalence for all small extensions, then 
$f$ has a section $s$ by Lemma \ref{snexists}. Thus $f^*:[Y,E]_Y\to [X,E]_Y$ has a retract $s^*$ for all quasi-smooth morphisms $p:E \to Y$, so is injective.
But note that $X^{\Delta^1}\by_{Y^{\Delta^1}}Y \to X\by_YX$ also satisfies the hypotheses of the lemma, so must have a section, giving a homotopy $h$ as above. 

Given $a: X \to E$ over $Y$, by functoriality of $(-)^{\Delta^1}$, we obtain a map 
\[
 a^{\Delta^1}: X^{\Delta^1}\by_{Y^{\Delta^1}}Y \to E^{\Delta^1}\by_{Y^{\Delta^1}}Y.   
\]
Composing with $h$ gives $X \to E^{\Delta^1}\by_{Y^{\Delta^1}}Y$ over $Y$, giving a homotopy between $a$ and $f^*s^*a$, so showing that
\[
f^*s^*:  [X,E]_Y \to [X,E]_Y
\]
is the identity.
\end{proof}

\begin{definition}\label{idef}
Define $I$ to be the set of morphisms in $sc\Sp$ of the form
$$
(X \otimes \Delta^n)\cup_{(X \otimes \pd \Delta^n)}(Y \otimes \pd \Delta^n) \to Y \otimes \Delta^n,
$$
for $n \ge 0$, and $X \into Y $ in $c\Sp$  dual to a small extension in $s\C_{\L}$.
\end{definition}

\begin{definition}\label{jdef}
Define $J$ to be the set of morphisms in $sc\Sp$ of the forms:
\begin{enumerate}
\item[($J_1$)] 
$
(X \otimes \Delta^n)\cup_{(X \otimes \pd \Delta^n)}(Y \otimes \pd \Delta^n) \to Y \otimes \Delta^n,
$
for $n \ge 0$, and $X \into Y $ in $c\Sp$  dual to an acyclic small extension in $s\C_{\L}$;

\item[($J_2$)] $
(X \otimes \Delta^n)\cup_{(X \otimes  \L^n_k)}(Y \otimes \L_k^n) \to Y \otimes \Delta^n,
$
for $n \ge k \ge 0$, and $X \into Y $ in $c\Sp$  dual to a small extension in $s\C_{\L}$.
\end{enumerate}
\end{definition}

\begin{lemma}\label{spsmall}
Every geometric cofibration in $sc\Sp$ is a relative $I$-cell complex, i.e. a transfinite composition of pushouts of elements of $I$.
\end{lemma}
\begin{proof}
Since every closed immersion in $c\Sp$ is a composition of small extensions,
$$
(X \otimes \Delta^n)\cup_{(X \otimes \pd \Delta^n)}(Y \otimes \pd \Delta^n) \to Y \otimes \Delta^n
$$
is a relative $I$-cell for all $X\into Y$ in $c\Sp$.

Take a geometric cofibration $f:X \to Y$ in $sc\Sp$, and 
consider the pushout diagram (in $c\Sp$)
$$
\xymatrix{
 (Y_n\otimes \pd \Delta^n)\cup_{(L_n(f)\otimes \pd \Delta^n)}(L_n(f) \otimes \Delta^n)        \ar[r] \ar[d] &   \sk^X_{n-1}Y      \ar[d]\\
     Y_n \otimes \Delta^n    \ar[r]  &  \sk^X_nY    }
$$
of \cite{sht} Proposition VII.1.9.
Since $Y= \varinjlim \sk^X_nY$, it suffices to show that
$$
(Y_n\otimes \pd \Delta^n)\cup_{(L_n(f)\otimes \pd \Delta^n)}(L_n(f) \otimes \Delta^n) \to Y_n \otimes \Delta^n
$$
is a relative $I$-cell. 

This, in turn, will follow if $L_n(f) \to Y_n$ is a  closed immersion in $c\Sp$. Now, 
$$
O(X)^n\cong O(L_nX) \oplus N_c^n(O(X)),
$$ 
and similarly for $Y$. Since $O(L_nf)=O(X)^n\by_{O(L_nX)}O(L_nY)$, we just require that $N_cO(Y) \to N_cO(X)$ be surjective, which is equivalent to $O(Y) \to O(X)$ being surjective, i.e. to $f$ being a geometric cofibration.
\end{proof}

\begin{lemma*}
Given a pushout $U \to V$ of a morphism in $J$, and a quasi-smooth map $f:X \to U$, there exists a quasi-smooth map $g:Y \to V$  and an isomorphism
$$
\phi:X \cong Y\by_VU
$$
over $U$. The pair $(Y,\phi)$ is unique up to unique isomorphism over $U$.
\end{lemma*}
\begin{proof}
Write $V= \Spf A$, $U=\Spf B$ and $X\by_U\Spf k = \Spf R$; set $M= \ker (A \to B)$.   Since $f$ is quasi-smooth, it is a Reedy fibration. Adapting  \cite{stacks2} 
Proposition \ref{stacks-deform1} to pro-Artinian rings, we see that the obstruction to lifting $f$ to a Reedy fibration $g: Y \to V$ lies in
$$
\H^2(X/U \hat{\ten} \Tot^{\Pi} N(M\hat{\ten}R)),
$$
with notation as in Lemma  \ref{tgtarb}. 
If the obstruction is zero, then the same proposition gives that the isomorphism class of liftings is  isomorphic to
$$
 \H^1(X/U \hat{\ten} \Tot^{\Pi} N(M\hat{\ten}R)) 
$$

Since $U \to V$ is a pushout  of a morphism in $J$, it follows that $\H_*(\Tot^{\Pi} NM)=0$ (as this is true  for morphisms in $J$). By the 
K\"unneth formula, 
$$
\H_*\Tot^{\Pi} N(M\hat{\ten}R))\cong \H_*(\Tot^{\Pi} NM)\hat{\ten} H_*( \Tot^{\Pi}NR), 
$$
so $\H^n(X/U \hat{\ten} \Tot^{\Pi} N(M\hat{\ten}R))=0$ for all $n$. 

Thus $f$ has a unique lift to a Reedy fibration $g: T \to V$.
By looking at cohomology, we see that any quasi-smooth deformation of a smooth map in $c\Sp$ is necessarily smooth. Thus the partial matching maps $Y_n \to M_{\L^n_k}Y\by_{M_{\L^n_k}V}V_n$ are smooth in $c\Sp$, so $g$ is quasi-smooth in $sc\Sp$.
 \end{proof}

\begin{lemma*}
Take a diagram $X \xra{f} Y \xra{g} Z$ in $sc\Sp$. If any two of $f,g,gf$ are geometric weak equivalences, then so is the third.
\end{lemma*}
\begin{proof}
If $f$ and one of $g,gf$ are geometric weak equivalences, then this is immediate. This leaves the case where $gf$ and $g$ are assumed geometric weak equivalences and we wish to show that $f$ is also.

Take a quasi-smooth map $p:E \to Y$; we need to show that $[Y,E]_Y \to [X,E]_Y$ is an isomorphism. If $g$ is quasi-smooth, then $gp$ is also quasi-smooth, and by hypothesis we have isomorphisms $[X,E]_Z \cong [Z,E]_Z \cong [Y,E]_Z$. Now, $g$ satisfies the conditions of Lemma \ref{smallworks}, 
hence so does $E\by_ZY \to E$. From the proof of Lemma \ref{smallworks}, 
it thus follows that $[Y,E\by_ZY]_Y \cong [Y,E]_Y$ and similarly for $X$, so the isomorphism above becomes $[X,E]_Y \cong [Y,E]_Y$, as required.

If $g$ is now any geometric weak equivalence, note that we may use the small object argument to factorise $g$ as $Y \xra{g'} Z' \xra{g''} Z$, where $g'$ is a relative  $J$-cell and $g''$ is in $J$-inj (i.e. quasi-smooth). Applying the result of the first paragraph to $g, g', g''$, we see that $g''$ must be a geometric weak equivalence. Applying the second paragraph to the diagram $ X \xra{g'f} Z' \xra{g''} Z  $, we deduce that $g'f$ must also be a geometric weak equivalence. Replacing $g$ by $g'$, we have therefore reduced to the case where $g$ is a relative  $J$-cell. We may apply the lemma above inductively to obtain a quasi-smooth map $\tilde{p}:\tilde{E} \to Z$ with $g^*\tilde{E}\cong E$. By hypothesis, we have isomorphisms $[X,\tilde{E}]_Z \cong [Z,\tilde{E}]_Z \cong [Y,\tilde{E}]_Z  $, or equivalently
$
  [X,E]_Y \cong  [Y,E]_Y,       
$
as required.
\end{proof}

\begin{theorem}\label{scspmodel}
There is a simplicial model structure, the ``geometric model structure'' on $sc\Sp$ with the cofibrations, fibrations and weak equivalences of Definition \ref{defcatdef}. 
It is cofibrantly generated, with $I$ the generating cofibrations, and $J$ the generating trivial cofibrations.
\end{theorem}
\begin{proof}
We verify the conditions of \cite{Hovey} Theorem 2.1.19.
\begin{enumerate}
\item by the lemma above,
the class of geometric weak equivalences  has the two out of three property; it is clearly closed under retracts.
\item[(2)--(3).]Note that the domains of $I$ and $J$ are small.
\item[(4).] It follows from Lemma \ref{spsmall} that $I$-cell is the class of geometric cofibrations; note that this is closed under retracts. It is immediate that $J$-cell is contained in the class of geometric trivial cofibrations.
\item[(5)--(6).]By definition, the geometric fibrations are precisely $J$-inj, and Lemma \ref{smallworks} implies that geometric trivial fibrations are precisely $I$-inj.
\end{enumerate}
Finally, it is an easy exercise to verify the simplicial model axiom (SM7a) (\cite{sht} \S II.3): that for any quasi-smooth map $q:X \to Y$, 
$$
X^{\Delta^n} \to X^{\pd \Delta^n}\by_{Y^{\pd \Delta^n}} Y^{\Delta^n} 
$$
is quasi-smooth, and a weak equivalence whenever $q$ is, and that
$$
X^{\Delta^1} \to X^{\{e\}} \by_{Y^{\{e\}}} Y^{\Delta^1}
$$ 
is a quasi-smooth weak equivalence for $e=0,1$.
\end{proof}

\begin{corollary}\label{homrep}
For $X \in sc\Sp$ and $A \in s\C_{\L}$,
$$
X(A) =\underline{\Hom}_{sc\Sp}(\Spf A, X) \in \bS. 
$$
\end{corollary}

\begin{corollary} \label{weakchar}
A morphism $f:X \to Y$ between quasi-smooth objects is a geometric weak equivalence if and only if it is a weak equivalence in the sense of Definition \ref{weakdef}. By Corollary \ref{weak}, this is equivalent to $\H^i(f):\H^i(X)\to\H^i(Y)$ being an isomorphism for all $i \in \Z$. 
\end{corollary}
\begin{proof}
If $f$ is a geometric weak equivalence, then Corollary \ref{homrep} implies that it must be a 
 weak equivalence in the sense of Definition \ref{weakdef}.

Given $U \in sc\Sp$, write $U=\Spf A$, for $A \in cs\hat{\C}_{\L}$.
Then
$$
\underline{\Hom}(U, X) = \{ x \in \prod_{n \in \N_0} X(A^n)^{\Delta^n}\,:\, \pd^i_Ax_n= (\pd^i)^*x_{n+1},\,\sigma^i_Ax_n= (\sigma^i)^*x_{n+1}\}.
$$

If $f$ is a weak equivalence in the sense of Definition \ref{weakdef}, then the maps $f:X(A^n) \to Y(A^n)$ are weak equivalences between fibrant simplicial sets for all $n$; it follows that
$$
f_*:\underline{\Hom}(U, X)\to \underline{\Hom}(U, Y)
$$
must also be a weak equivalence between fibrant simplicial sets. Since 
$$
\Hom_{\Ho(sc\Sp)}(U,X)= \pi_0\underline{\Hom}(U, X)
$$
and $U$ was arbitrary, $f$ must be a geometric weak equivalence.
\end{proof}

Lemma \ref{settotop} now implies:
\begin{lemma}\label{underlineq}
The functor from $c\Sp$ to $sc\Sp$ given by $X \mapsto \underline{X}$ is simplicial right Quillen. 
\end{lemma}

\subsubsection{Representing cohomology}\label{srepcoho}

\begin{definition}\label{eilmac}
For $n\ge 0$ define $K(n) \in sc\Sp$ to be the object $\Spf (k \oplus K^n\eps) \in c\Sp$, for $K^n$ as defined in \S \ref{cohomology}. For $n\le 0$,  define $K(n) \in sc\Sp$ to be 
$$
 (\Spf k[\eps] \otimes{\Delta^{-n}})\cup_{(\Spf k[\eps]\otimes{\pd \Delta^{-n}})} \Spf k \in s\Sp.
$$
\end{definition}

\begin{definition}
Given $Z \in sc\Sp$ and  $X, Y \in sc\Sp\da Z$, define $[X,Y]_Z:=\Hom_{\Ho(cs\Sp\da Z)}(X,Y)$. Note that when $Y \to Z$ is a fibration, this agrees with Definition \ref{Xpdef}.
\end{definition}

\begin{lemma}\label{repcoho}
For $X\to Z$ quasi-smooth, $\H^n(X/Z)=[K(n), X]_Z$.
\end{lemma}
\begin{proof}
Since $X$ is fibrant in $cs\Sp \da Z$, and $ K(n)$ cofibrant, with $X \to X^{\Delta^1}\by_{Z^{\Delta^1}}Z\to X\by_Z X$ a path object, we have a coequaliser diagram
$$
\xymatrix@1{\Hom( K(n),X^{\Delta^1}\by_{Z^{\Delta^1}}Z)_Z \ar@<.5ex>[r]\ar@<-.5ex>[r] &\Hom( K(n),X)_Z \ar[r] &[ K(n), X]_Z}.
$$

For $n \ge 0$, this is just
$$
\xymatrix@1{F_1(K(n)) \ar@<.5ex>[r]\ar@<-.5ex>[r] & F_0(K(n)) \ar[r] &[ K(n), X]_Z},
$$
for $F$ the fibre of $X \to Z$ over the initial object. Thus
$$
[ K(n), X]_Z =\pi_0 (F(K(n)))= \H^n(X/Z).
$$

For $n\le 0$ a similar argument gives 
$$
[ K(n), X]_Z =\pi_{-n} (F(k[\eps])= \H^n(X/Z).
$$
\end{proof}

\begin{definition}\label{generalcoho}
Given any morphism $f:X \to Z$, we define $\H^n(X/Z):= [K(n),X]_Z$, or equivalently $\H^n(X,Z):=\H^n(\hat{X}/Z)$, for $X \xra{i} \hat{X} \xra{p} Z$ a factorisation of $f$ with $i$ a geometric trivial cofibration, and $p$ a geometric fibration. It follows from Lemma \ref{repcoho} that this is well-defined.
\end{definition}

\subsubsection{Comparison with the Reedy model structure}\label{reedy}

\begin{definition}
Define $I_R$ to be the set of morphisms in $sc\Sp$ of the form
$$
(X \otimes \Delta^n)\cup_{(X \otimes \pd \Delta^n)}(Y \otimes \pd \Delta^n) \to Y \otimes \Delta^n,
$$
for $n \ge 0$, and $X \into Y $ in $I_{\Sp}$ (i.e. a morphism in $c\Sp$ either dual to a small extension in $s\C_{\L}$, or an arbitrary map in $\Sp$).
\end{definition}

\begin{definition}
Define $J_R$ to be the set of morphisms in $sc\Sp$ of the form
 $$
(X \otimes \Delta^n)\cup_{(X \otimes \pd \Delta^n)}(Y \otimes \pd \Delta^n) \to Y \otimes \Delta^n,
$$
for $n \ge 0$ and $X \into Y $ in $J_{\Sp}$ (i.e. a morphism  in $c\Sp$  dual to an acyclic small extension in $s\C_{\L}$).
\end{definition}

\begin{definition}
Recall that the model structure on $c\Sp$ gives rise to a Reedy model structure on $sc\Sp$, for which $I_R$ is the class of generating cofibrations, and $J_R$ the class of generating trivial cofibrations. 
\end{definition}

\begin{lemma}\label{reedycomp}
Every Reedy trivial cofibration is a geometric trivial cofibration, and every Reedy trivial fibration is a geometric trivial fibration. Thus every Reedy weak equivalence is a geometric weak equivalence. Conversely, every  geometric fibration (resp. cofibration) is a Reedy fibration (resp. cofibration).
\end{lemma}
\begin{proof}
Observe that $J_R =J_1\subset J$, so $J_R$-cof $\subset J$-cof, and that $I \subset I_R$, so $I_R$-inj $\subset I$-inj.
\end{proof}

\begin{lemma}\label{levelwiseqs}
 Let $X \in sc\Sp$ be levelwise quasi-smooth, in the sense that each $X_n \in c\Sp$ is quasi-smooth. Then the canonical map $X \to \underline{X}$ is a geometric weak equivalence.
\end{lemma}
\begin{proof}
At simplicial level $n$, this map is just $f_n:X_n \to X_n^{\Delta^n}$ in $c\Sp$, in the notation of the simplicial model structure of Definition \ref{smcldef}. Since $X$ is fibrant in $c\Sp$, $f_n$ is a weak equivalence in $c\Sp$, so $f$ is a Reedy weak equivalence. 
\end{proof}

\begin{lemma}
For all quasi-smooth $X \in c\Sp$, the canonical map $X \to \underline{X}$ is a fibrant approximation of $X$ in the geometric model structure on $sc\Sp$.
\end{lemma}
\begin{proof}
By Lemma \ref{settotop}, we already know that $\underline{X}$ is quasi-smooth, and we have just seen that $f:X \to \underline{X}$ is a geometric weak equivalence. 
\end{proof}

\subsection{Homotopy representability}\label{snschless}

\begin{definition}\label{schless}
Define the category $\cS$ to consist of functors $F: s\C_{\L}\to \bS$ satisfying the following conditions:
\begin{enumerate}

\item[(A0)] $F(k)$ is contractible.

\item[(A1)] For all small extensions $A \onto B$ in $s\C_{\L}$, and maps $C \to B$ in $s\C_{\L}$, the 
map 
$F(A\by_BC) \to F(A)\by_{F(B)}^hF(C)$ is a  
weak equivalence, where $\by^h$ denotes homotopy fibre product.

\item[(A2)] For all acyclic small extensions $A \onto B$ in $s\C_{\L}$, the map $F(A) \to F(B)$ is a weak equivalence.
\end{enumerate}

Say that a natural transformation $\eta:F \to G$ between such functors is a weak equivalence if the maps $F(A) \to G(A)$ are weak equivalences for all $A\in s\C_{\L}$, and let $\Ho(\cS)$ be the category obtained by formally inverting all weak equivalences in $\cS$.
\end{definition}

\begin{remark}\label{byh}
We may apply the long exact sequence of homotopy  to describe the homotopy groups of homotopy fibre products. 
If $f:X \to Z$, $g:Y \to Z$ in $\bS$ and $P=X\by_Z^hY$,   the map $\theta:\pi_0(P)\to \pi_0(X)\by_{\pi_0(Z)}\pi_0(Y)$ is surjective. Moreover, $\pi_1(Z,*)$ acts transitively on the fibres of $\theta$ over $* \in \pi_0Z$.

Take $v \in \pi_0(P)$ over $*$. Then 
there is
a connecting homomorphism $\pd:\pi_n(Z,*)\to \pi_{n-1}(P,v)$ for all $n\ge 1$, 
giving a long exact sequence
$$
\ldots \xra{\pd} \pi_n(P,v)\to \pi_n(X,v)\by \pi_n(Y,v)\xra{f\cdot g^{-1}} \pi_n(Z,*)\xra{\pd}\pi_{n-1}(P,v)  \ldots .
$$
\end{remark}

The following can be regarded as an analogue of Schlessinger's theorem (\cite{Sch} Theorem 2.11), or as a Brown-type representability theorem with (A1) the Mayer-Vietoris condition. 
\begin{theorem}\label{schrep}
There is a canonical equivalence between the geometric homotopy category $\Ho(sc\Sp)$ and the category $\Ho(\cS)$.
\end{theorem}
\begin{proof}
Given a quasi-smooth object $X \in sc\Sp$, observe that the functor $\theta(X)$ on $s\C_{\L}$ given  by $A \mapsto \underline{\Hom}(\Spf A,X)$ satisfies (A0)--(A2), and that Corollary \ref{weakchar} implies that this construction descends to a functor $\theta:\Ho(sc\Sp) \to \Ho(\cS)$.

Conversely, given $F \in \cS$, we first extend $F$ to $s\hat{\C}_{\L}$: any $A \in s\hat{\C}_{\L}$ is isomorphic to an inverse system $\{A_{\alpha}\}$ indexed by a totally ordered set, with all transition maps surjective in $s\C_{\L}$, and we set $ F(A):= \holim_{\alpha} F(A_{\alpha})$. 

For $K \in \bS$, the endofunctor $X \mapsto X^K$ of $\bS$ is right Quillen; choose an associated derived right Quillen functor $X \mapsto X^{\bR K}$ (given by $(X^f)^K$, for $X^f$ a fibrant replacement of $X$).
We wish to define a functor $\overline{F}: (sc\Sp)^{\opp}\to \bS$ satisfying  $\overline{F}(U\ten K):= \overline{F}(U)^{\bR K}$ and  preserving  homotopy colimits. 

Given $U \in sc\Sp$, we now consider the simplicial skeleta
$$
U= \varinjlim \sk_nU,
$$
where $\sk_0= U_0 \in c\Sp \subset sc\Sp$, and $\sk_nU$ is given by  the pushout
$$
\xymatrix{
    \sk_{n-1}U     \ar[r] \ar[d] &     \sk_{n}U    \ar[d]\\
   {\Delta^n}\ten L_n U\by_{{\pd\Delta^n}\ten L_nU}\pd\Delta^n\ten U_n      \ar[r]  &  \Delta^n\ten U_n.}
$$

We may therefore define $\overline{F}(\sk_n U)$ inductively as the homotopy pullback
$$
\xymatrix{
     \overline{F}(\sk_n U)    \ar[r] \ar[d] &  \overline{F}(\sk_{n-1}U )      \ar[d]\\
   F(L_nU)^{\bR\Delta^n}\by_{F(L_nU)^{\bR\pd \Delta^n}}F(U_n)^{\bR\pd\Delta^n}      \ar[r]  & F(U_n)^{\bR\Delta^n}.}
$$

Now, it is straightforward to see that $\overline{F}$ maps morphisms in $J$ to weak equivalences, so it maps all trivial cofibrations to weak equivalences by Theorem \ref{scspmodel}. Given a weak equivalence $f:A \to B$ in $cs\hat{\C}_{\L}$, observe that the object $A\by_{f,B, \ev_0}B^{\Delta^1}$, dual to the  mapping cylinder, is equipped with
 trivial fibrations to both $A$ and $B$.
Hence  $\overline{F}$ descends to a functor $\overline{F}:\Ho(sc\Sp)^{\opp}\to \Ho(\bS)$. It is also easy to see that $\overline{F}$ preserves all homotopy limits.

Therefore the functor $\pi_0\overline{F}:\Ho(sc\Sp)^{\opp}\to \Set$ is half-exact in the sense of \cite{heller}, and  $\Ho(sc\Sp)$ satisfies the conditions of Heller's Theorem (\cite{heller} Theorem 1.3), so $\pi_0\overline{F}$ is representable, and we have defined a functor  $\cS\to \Ho(sc\Sp)$. If $\eta: F\to G$ is a weak equivalence in $\cS$, then  $\overline{F}(A) \to \overline{G}(A)$ is a weak equivalence for all $A$, so our functor descends to a functor $\Ho(\cS) \to \Ho(sc\Sp)$.

To see that these functors form a quasi-inverse pair, note that, for $K \in \bS$, $[K,F(A)]= \pi_0(\overline{F}(A)^{\bR K})= \pi_0(\overline{F}(\Spf A \ten K))$. Conversely, it is immediate that for a quasi-smooth $X \in sc\Sp$,  $\overline{\theta(X)}=\underline{\Hom}(-,X)$, so $\pi_0\overline{\theta(X)}=\Hom_{\Ho(sc\Sp)}(-,X)$.
\end{proof}

\begin{remark}\label{sgpd}
Since the homotopy categories of simplicial groupoids and simplicial sets are Quillen-equivalent (\cite{sht} Corollary V.7.11), this recovers the conception of extended deformation functors as taking values in simplicial groupoids.
\end{remark}

\subsection{Minimal models}\label{minimal}

\begin{definition}\label{dgstuff}
Given an abelian category $\cA$, let $dg\cA$ be the category of  non-negatively graded chain complexes in  $\cA$, and $dg_{\Z}\cA$  the category of  $\Z$-graded chain complexes in  $\cA$. Let $DG\cA$ be the category of  non-negatively graded cochain complexes in  $\cA$.
\end{definition}

\begin{definition}\label{totprod}
Define the total complex  functor $\Tot^{\Pi}:DGdg\widehat{\FD\Vect}_k \to  dg_{\Z}\widehat{\FD\Vect}_k$ by
$$
(\Tot^{\Pi} V)_n := \prod_{a-b=n} V^b_a,
$$
with differential $d=d_c +(-1)^b d^s$.
\end{definition}

\begin{definition}\label{tot*}
Let $\Tot^{\Pi*}: dg_{\Z} \widehat{\FD\Vect} \to DGdg\widehat{\FD\Vect}$  be left adjoint to $\Tot^{\Pi}$. Explicitly
$$
\Tot^{\Pi*}(V)^b_a=\left\{ \begin{matrix} V_{a-b} \oplus V_{a-b+1} & b>0 \\V_{a} & b=0,   \end{matrix} \right.
$$
with differentials $d_c(v,w)=(0, v)$, $d^s(v,w)= \pm(dv,v-dw)$.
\end{definition}

\begin{definition}
Say that a quasi-smooth object $R$ of $cs\hat{\C}_{\L}$ is \emph{minimal} if the cochain chain complex $N\cot R$ is of the form $ \Tot^{\Pi*}(V_*)$, for a $\Z$-graded vector space $V_*$ (regarded as a chain complex with zero differential). 
\end{definition}

Every cochain complex over a field is homotopy-equivalent to its cohomology. This has the following trivial corollary, which we regard as the analogous statement for chain complexes of cochain complexes:

\begin{lemma}\label{dcoch}
Let $ \ldots \xra{\delta} V_2 \xra{\delta} V_1 \xra{\delta} V_0  $ be a chain complex of cochain complexes. Then $V_{\bullet}$ is levelwise homotopy-equivalent to the chain complex 
$$
h_n^i(V) := \H^i(\delta V_{n+1})\oplus \H^i(V_n/\delta V_{n+1})
$$
of cochain complexes, with  $\delta(v,w)=(\delta w, 0)$, and $d(v,w)=(\pd w, 0) $, for $\pd: \H^i(V_n/\delta V_{n+1})\to\H^{i+1}(\delta V_{n+1})$ the boundary map associated to the short exact sequence $0 \to \delta V_{n+1}\to V^n\to V^n/\delta V^{n+1}\to 0$. 
\end{lemma}

\begin{lemma}\label{decompcot}
Given $V\in DGdg\widehat{\FD\Vect}$ quasi-smooth (in the sense of Definition \ref{qsvect}), there exists a decomposition  
$$
V \cong U \oplus \Tot^{\Pi*}(\H_*(\Tot^{\Pi}V)),
$$
of cochain chain complexes, with $\Tot^{\Pi}U$ acyclic.
\end{lemma}
\begin{proof}
Let $T:=\Tot^{\Pi}V$ and   $W:=\Tot^{\Pi*}(\H_*(T))$. Recall that the conditions for $V$ to be quasi-smooth are that $\H^i(V_n)=0$ for all $i,n \ge 0$, and that $\H_n(V^i)=0$ for all $i,n>0$. 

By Lemma \ref{dcoch} there is a levelwise cochain homotopy equivalence between $V$ and
$$
\cH(V)^i_n:= \H^i(d^sV_{n+1}) \oplus \H^i(V_n/d^sV_{n+1}),
$$
with $d_c(x,y)= (\pd y,0),\, d^s(x,y)=(d^sy,0)$. In particular, this makes $\cH(V)$ a direct summand of $V$.

Since $V$ is quasi-smooth, $\pd:\H^i(V_n/d^sV_{n+1})\to \H^{i+1}(V_n)$ is an isomorphism, and both groups are isomorphic to $\H_{n-i}(T)$. Thus $\cH(V)\cong W$, and $\Tot^{\Pi}U$ is necessarily acyclic, since $\Tot^{\Pi}V\to \Tot^{\Pi}W$ is a quasi-isomorphism.
\end{proof}

\begin{proposition}
Every weak equivalence class in $cs\hat{\C}_{\L}$ has a minimal model, unique up to non-unique isomorphism.
\end{proposition}
\begin{proof}
Choose a quasi-smooth representative $R$ in the weak equivalence class. Working inductively on the cochain degree, we may choose a decomposition 
$$
N\cot R \cong U_{\bt}^{\bt} \oplus \Tot^{\Pi*}(\H_*(\Tot^{\Pi}N\cot R))
$$
of cochain chain complexes over $k$, as in Lemma \ref{decompcot}. Observe that 
 $U_{\bt}^{\bt}$ is quasi-smooth and that $\H_*\Tot^{\Pi}U=0$. 

Since these conditions are equivalent to saying that the rows and columns of $U_{\bt}^{\bt}$ are all acyclic, working inductively we can lift  $U_{\bt}^{\bt}$ to an acyclic cochain chain complex  $\tilde{U}_{\bt}^{\bt}$ of free pro-Artinian $\L$-modules. As  $\L[[N^{-1}\tilde{U}]]$ is then trivially cofibrant in $cs\hat{\C}_{\L}$, the map $\L[[N^{-1}\tilde{U}]] \to k \oplus (\cot R)\eps$ lifts to  a map
$
\L[[N^{-1}\tilde{U}]]\xra{f} R;
$
define $S:=R/\langle f(N^{-1}\tilde{U})\rangle$. $S$ is levelwise smooth, with cotangent space $N^{-1} \Tot^{\Pi*}(\H_*(\Tot^{\Pi}N\cot R))$, so it must be quasi-smooth and minimal. This proves existence.

For uniqueness, observe that if $T$ is another minimal model in the same equivalence class, there must exist a weak equivalence
$$
f:S \to T,
$$
$S$ being cofibrant and $T$ fibrant. By the minimality criterion, $\cot f: \cot S \to \cot T$
must then be an isomorphism. Thus $f^n_i:S^n_i \to T^n_i$ must be an isomorphism for all $i,n$, as the isomorphism on cotangent spaces induces an isomorphism of the associated graded rings.  
\end{proof}

\subsection{Characterising trivial small extensions}

We end this section with a result which will help to give a more concrete description of geometric trivial cofibrations in $sc\Sp$.

\begin{definition}
Given a bounded complex $V  \in dg_{\Z}\FD\Vect_k$ (notation as in Definition \ref{dgstuff}) and $F\to G$ a quasi-smooth morphism in $sc\Sp$, set
$$
\H^n(F/G\ten V) :=\bigoplus_{i-j=n} \H^i(F/G) \ten \H_j(V). 
$$

Given a pro-object $V=\{V_{\alpha}\}  \in dg_{\Z}\widehat{\FD\Vect}_k$, for $V_{\alpha}$ finite-dimensional, set
$\H^n(F/G\hat{\ten} V):= \Lim \H^n(F/G\ten V_{\alpha})$. Note that we then have an isomorphism
$$
\H^n(F/G\hat{\ten} V)\cong \prod_{i \in \Z}\Hom(\H_i(V)^{\vee}, \H^{n+i}(F/G)).
$$
\end{definition}

\begin{lemma}\label{tgtarb}
Given $V  \in cs\widehat{\FD\Vect_k}$,  and $X \to Z$ quasi-smooth in $sc\Sp$, there is a canonical isomorphism 
$$
\pi_0\underline{\Hom}(\Spf (k \oplus V\eps), X)_Z \cong \H^0(X/Z\hat{\ten} \Tot^{\Pi} NV),
$$
for $\Tot^{\Pi}$ as in Definition \ref{totprod}.
\end{lemma}
\begin{proof}
First assume that $V  \in cs\FD\Vect_k$, with bounded binormalisation $NV=N^sN_cV$.

Given $W \in s\FD\Vect_k$ and $K \in \bS$, define $(K,W)\in cs\FD\Vect_k$ by $(K,W)^n:= W^{K_n}$. We may now express $V$ in terms of cosimplicial coskeleta by
$$
V= \Lim \cosk_nV,
$$
with $\cosk_0V= V^0 \in s\FD\Vect_k$, and $\cosk_nV$ given by the pullback
$$
\xymatrix{
     \cosk_nV    \ar[r] \ar[d] &   \cosk_{n-1}V      \ar[d]\\
   (\Delta^n,V^n)      \ar[r]  &  ({\Delta^n},M^nV)\by_{({\pd\Delta^n},M^nV)}({\pd\Delta^n},V^n).    }
$$
Since $N^n_cV=\ker(V^n \to M^nV)$, the kernel of $\cosk_nV \to \cosk_{n-1}V$ is thus $(S^n,N^n_cV):=  \ker((\Delta^n, N^n_cV) \to ({\pd\Delta^n},N^n_cV))$. 

If we write $Y(A):= \underline{\Hom}(\Spf A, X)_Z \in \bS$ for $A \in cs\hat{\C}_k$, then $Y(\cosk_nV)$ forms a tower of fibrations, with fibres $\Omega^nY(N^n_cV):= \ker (Y(N^n_cV)^{\Delta^n}\to Y(N^n_cV)^{\pd\Delta^n})$. This gives us a spectral sequence
$$
E_1^{n,m}=\pi_{m-n}\Omega^nY(N^n_cV) \abuts \pi_{m-n}Y(V),
$$
which converges since $N_c^nV=0$ for $n \gg 0$.

There are canonical isomorphisms $\pi_{m-n}\Omega^nY(N^n_cV)\cong \pi_{m}Y(N^n_cV) \cong \H^{-m}(X/Z \ten N^n_cV)$. Calculation of  the differentials  shows that this spectral sequence is isomorphic to the spectral sequence
$$
E_1^{n,m}= \H^{-m}(X/Z \ten N^n_cV) \abuts \H^{n-m}(X/Z \ten \Tot NV),
$$
associated to the double complex $NV$. 

Thus
$$
\pi_0Y(V)\cong \H^0(X/Z\ten \Tot NV).
$$
For the general case, write $V=\Lim V_{\alpha}$, for  $V_{\alpha}  \in cs\FD\Vect_k$ with bounded binormalisation. Then
$$
\pi_0Y(V)=\Lim \pi_0Y(V_{\alpha})\cong \Lim\H^0(X/Z\ten \Tot NV_{\alpha})= \H^0(X/Z\hat{\ten} \Tot^{\Pi} NV),
$$
as required.
\end{proof}

\begin{definition}
Define a small extension  in $cs\hat{\C}_{\L}$ to be a surjection $A \to B$  with kernel $I$, such that $\m(A)\cdot I=0$. 
\end{definition}

\begin{lemma}\label{weakext}
A small extension $f:A \to B$ in $cs\hat{\C}_{\L}$, with kernel $I$, is a weak equivalence if and only if $\H_*(\Tot^{\Pi} NI)=0$.
\end{lemma}
\begin{proof}
Taking the cone $C$ of  $I \to A$ as in Theorem \ref{robs} and a quasi-smooth morphism $X \to Z$, we get a fibration sequence
$$
\underline{\Hom}(\Spf A, X)_Z \to \underline{\Hom}(\Spf C, X)_Z\to \underline{\Hom}(\Spf (k \oplus I[-1]\eps), X)_Z,
$$
with $\underline{\Hom}(\Spf C, X)_Z \to\underline{\Hom}(\Spf B, X)_Z$ a weak equivalence.

Now, $\Hom(\Spf A, X)_Z \to \Hom(\Spf A, X)_Z$ is surjective if and only if 
the fibration $\underline{\Hom}(\Spf A, X)_Z \to \underline{\Hom}(\Spf B, X)_Z$  is surjective on $\pi_0$. The long exact sequence associated to a fibration implies that this automatically occurs whenever
$$
\pi_0\underline{\Hom}(\Spf (k \oplus I[-1]\eps), X)_Z=0.
$$
By Lemma \ref{tgtarb}, this is isomorphic to $\H^1(X/Z \hat{\ten}\Tot^{\Pi} NI)=0$, so the condition is sufficient. 

For necessity, observe that the condition is satisfied by morphisms in $J$ (as in Definition \ref{jdef}), and  recall that every weak equivalence is a relative  $J$-cell.
\end{proof}

\section{Other formulations of derived deformation theory}\label{back}

\subsection{Manetti's deformation functors}

The results in this section all come from \cite{Man2}.

\subsubsection{DGLAs}

\begin{definition}\label{dgzsp}
Define $dg_{\Z}\C_{\L}$ to be the category of Artinian local differential $\Z$-graded  graded-commutative  $\L$-algebras with residue field $k$.  Let  $dg_{\Z}\hat{\C}_{\L}$ be the category of pro-objects of $dg_{\Z}\C_{\L}$. Denote the opposite category $(dg_{\Z}\hat{\C}_{\L})^{\opp}$ by $DG_{\Z}\Sp$. Given $R \in dg_{\Z}\hat{\C}_{\L}$, let $\Spf R \in DG_{\Z}\Sp $ denote the corresponding object in the opposite category. 
\end{definition}

\begin{remark}
The category $dg_{\Z}\C_{k}$ is equivalent to the category $\mathbf{C}$ of \cite{Man2}, with $A \in dg_{\Z}\C_{k}$ corresponding to $C \in \mathbf{C}$ given by $C_n:= \m(A)_{-n}$.
\end{remark}

\begin{definition}
Define a surjective map $f: A \to B$ in $dg_{\Z}\C_{\L}$ to be a small extension if it is surjective   with kernel $V$, such that $\m(A)\cdot V=0$.
\end{definition}

For the rest of this section, assume that $\L=k$, a  field   of characteristic $0$.

\begin{definition} As in \cite{Man2} Definition 2.14, a DGLA over $k$  is a  graded vector space $L=\bigoplus_{i \in \Z} L^i$ over $k$, equipped with operators $[,]:L \by L \ra L$ bilinear and $d:L \ra L$ linear,  satisfying:

\begin{enumerate}
\item $[L^i,L^j] \subset L^{i+j}.$

\item $[a,b]+(-1)^{\bar{a}\bar{b}}[b,a]=0$.

\item $(-1)^{\bar{c}\bar{a}}[a,[b,c]]+ (-1)^{\bar{a}\bar{b}}[b,[c,a]]+ (-1)^{\bar{b}\bar{c}}[c,[a,b]]=0$.

\item $d(L^i) \subset L^{i+1}$.

\item $d \circ d =0.$

\item $d[a,b] = [da,b] +(-1)^{\bar{a}}[a,db]$
\end{enumerate}

Here $\bar{a}$ denotes the degree of $a$, mod $ 2$, for $a$ homogeneous.
\end{definition}

\begin{definition}\label{dglacat}
Define $DG_{\Z}\LA$ to be the category of differential $\Z$-graded Lie algebras $L^{\bt}$ over $k$ 
\end{definition}

\begin{definition}\label{mcdef}
Given a DGLA $L$ over $k$, the Maurer-Cartan functor $\mc(L): dg_{\Z}\C_{k}\to \Set$ is defined by
$$
\mc(L)(A):= \{\omega \in \bigoplus_n L^{n+1}\ten \m(A)_n \,|\, d\omega + \half[\omega,\omega]=0 \in  \bigoplus_n L^{n+2}\ten \m(A)_n\}.
$$
 where $\m(A)$ is the maximal ideal of $A$.

There is a gauge action of the group $\exp( \bigoplus_n L^{n}\ten \m(A)_n)$ on $\mc(L)(A)$; denote the quotient set by $\ddef(L)(A)$.
\end{definition}

\begin{definition}\label{mandef}
A functor
$F:dg_{\Z}\C_{\L}\to \Set$ is  a ``predeformation functor" in the sense of \cite{Man2} Definition 2.1 if:
\begin{enumerate}
\item[(A0)] $F(k)=\bt$.

\item[(A1)] For all small extensions $A \onto B$, and morphisms $C \to B$ in $dg_{\Z}\C_{\L} $, the  map 
$$
F(A\by_BC)\to F(A)\by_{F(B)}F(C)
$$ is surjective. It is an isomorphism whenever $B \simeq k$.

\item[(A2')] For all acyclic small extensions $A \onto B$ in $dg_{\Z}\C_{\L}$, the map $F(A) \to F(B)$ is a surjection.
\end{enumerate}

It is a ``deformation functor'' if in addition

\begin{enumerate}
\item[(A2)] For all acyclic small extensions $A \onto B$ in $dg_{\Z}\C_{\L}$, the map $F(A) \to F(B)$ is an isomorphism.
\end{enumerate}
\end{definition}

\begin{lemma}
For any predeformation functor $F$, there exists a deformation functor $F^+$, and a natural transformation $F \to F^+$, universal among transformations from $F$ to deformation functors. 
\end{lemma}
\begin{proof}
\cite{Man2} Theorem 2.8. 
\end{proof}

\begin{lemma}
For any DGLA $L$, $\mc(L)$ is a predeformation functor, $\ddef(L)$ is a deformation functor, and $\ddef(L)\cong \mc(L)^+$.
\end{lemma}
\begin{proof}
\cite{Man2}  Lemma 2.17,  Theorem 2.19 and Corollary 3.4.
\end{proof}

\subsubsection{SHLAs}

\begin{definition}
A  graded coalgebra is a $\Z$-graded vector space $C$ equipped with a (graded-)cocommutative coassociative comultiplication $C\to C\ten C$. A dg coalgebra is a graded coalgebra equipped with a square-zero degree $1$ codifferential $d$, compatible with the comultiplication (making $d$ into a coderivation).
\end{definition}

\begin{definition}
Let $\Gamma^nV$ be the $S_n$-invariants (with respect to the usual graded convention) of the tensor power $V^{\ten n}$. Thus $\Gamma^nV \cong S^nV$, the $S_n$-covariants, since we are working in characteristic $0$. 

Given a graded vector space $V$, define $C(V)$ to be the cofree (ind-conilpotent) graded coalgebra $C(V):= \bigoplus_{n>0} \Gamma^nV$ given by the graded symmetric powers of $V$. A comultiplication is defined on $F(V):=\bigoplus_{n>0} V^{\ten n}$ 
by mapping 
$v_1\ten v_2\ldots \ten v_n \in V^{\ten n}$ to $\sum_{i=1}^{n-1} (v_1\ten \ldots \ten v_i) \ten (v_{i+1}\ten \ldots \ten v_n) \in  F(V)\ten F(V)$. The restriction of this comultiplication to $C(V)$ is cocommutative.
\end{definition}

\begin{remark}
A coalgebra $C$ is conilpotent if the iterated comultiplication $\Delta_n :C \to C^{\ten n}$ is $0$ for $n \gg 0$. A coalgebra is ind-conilpotent if it is the union of its conilpotent subcoalgebras. The functor $V \mapsto C(V)$ is right adjoint to the forgetful functor from graded  ind-conilpotent coalgebras to graded vector spaces.

In \cite{Kon}, $V \mapsto C(V)$ is referred to as the cofree coalgebra functor. This is misleading, since it is not right adjoint to the   forgetful functor from all  coalgebras to vector spaces. This right adjoint (the true cofree coalgebra functor) is very difficult to describe explicitly (see \cite{Sweedler}), but will not concern us here. 
\end{remark}

\begin{definition}\label{linfty}
An $L_{\infty}$ structure on a $\Z$-graded vector space $V$ is a codifferential $d$ on the graded coalgebra $C(V[1])$, making $C(V[1])$ into a dg coalgebra. The space $V$ together with its $L_{\infty}$ structure is called an $L_{\infty}$-algebra.

We will follow \cite{Kon} in saying that
an SHLA is a dg coalgebra whose underlying graded coalgebra is isomorphic to $C(V)$, for some $V$. Thus an $L_{\infty}$-algebra is a choice of co-ordinates on an SHLA.
\end{definition}

\begin{lemma}\label{dglainfty}
Given a DGLA $L$, there is a natural $L_{\infty}$ structure on $L$.
\end{lemma}
\begin{proof}
On cogenerators $L[1]$, define  the coderivation on $C(L[1])$ to be the map $d_C:C(L[1])\to L[1]$ given by
$$
d_C( v_1\ten v_2\ldots \ten v_n) = \left\{\begin{matrix} dv_1 & n=1 \\ [v_1, v_2] & n=2 \\ 0 & n>2. \end{matrix} \right.
$$
\end{proof}

\begin{remark}\label{profd2}
Any SHLA $C:=C(W)$ can be written as a filtered direct limit $C(W)= \LLim_m\bigoplus_{0<n\le m} \Gamma^nW$ of subcoalgebras, and these subcoalgebras are conilpotent (this is what it means for $C$ to be ind-conilpotent). Now, every coalgebra is the union of its finite-dimensional subcoalgebras, so we can express $C$ as a filtered direct limit of finite-dimensional conilpotent coalgebras. Therefore the dual $C^{\vee}$ is a filtered inverse limit of finite-dimensional nilpotent commutative algebras without unit, so we may regard $k \oplus C^{\vee}$ as an object of $dg_{\Z}\hat{\C}_k$.

In fact, Lemma \ref{profd} implies that this construction gives a contravariant equivalence between $dg_{\Z}\hat{\C}_k$ and the category of (not necessarily cofree) conilpotent dg coalgebras. Explicitly, the quasi-inverse sends $A\in dg_{\Z} \hat{\C}_k$ to $\m(A)^{\vee}$, where $(\{V_{\alpha}\}_{\alpha\in I})^{\vee}:= \LLim_I (V_{\alpha})^{\vee}$ (the continuous dual).
\end{remark}

\begin{definition}
Given an $L_{\infty}$-algebra $V$, write $C:=C(V[1])$ with its dg coalgebra structure, and   define $\mc(V):dg_{\Z}\C_{k}\to \Set$ by
$$
\mc(V)(A):= \Hom_{dg_{\Z}\hat{\C}_k}(k\oplus C^{\vee}, A).
$$
\end{definition}
In \cite{Man2} \S 5, this definition is phrased in the opposite category (as $\Hom(\m(A)^{\vee}, C)$).

 By \cite{Man2} Proposition 4.5, $\mc(V)$ is a predeformation functor. Note that if $V$ is the $L_{\infty}$-structure associated to a DGLA $L$, then $\mc(V) \cong \mc(L)$. 

\begin{definition}
Given an $L_{\infty}$-algebra $V$, define $\ddef(V):dg_{\Z}\C_{k}\to \Set$ by $\ddef(V):= \mc(V)^+$.
\end{definition}

\subsection{Hinich's formal stacks}

We begin with some  properties of SHLAs from \cite{Kon}.
\begin{definition}
Given a dg coalgebra (for example an SHLA) $C$, define $\tan(C)$ to be the kernel of the comultiplication $\Delta: C \to C\ten C$. Note that this is a cochain complex, and that if $C$ is an $L_{\infty}$-structure on a graded vector space $V$, then $\tan(C)\cong V$. If the $L_{\infty}$-structure comes from a DGLA $L$, then $\tan(C)$ is just the cochain complex underlying $L$.
\end{definition}

\begin{definition}\label{tqis}
Say that a morphism $f:C \to D$ of SHLAs is a tangent quasi-isomorphism if the associated map $\tan(f): \tan(C) \to \tan(D)$ is a quasi-isomorphism of cochain complexes. Note that this is a stronger condition than $f$ being a quasi-isomorphism.
\end{definition}

\begin{definition}
Define $dg\C_{\L}$ to be the category of Artinian local differential $\N_0$-graded  graded-commutative  $\L$-algebras with residue field $k$.
\end{definition}

\begin{remark}
In \cite{hinstack}, the category $dg\C_{\L}$ is denoted by $dg\Art_{\L}^{\le 0}$, with $A \in dg\C_{\L}$ corresponding to $C \in dg\Art_{\L}^{\le 0}$ given by $C^n:= A_{-n}$.
\end{remark}

Now let $\L=k$, a field of characteristic $0$.

\begin{definition}
Let  $DG_{\Z}\mathrm{CU}_k$ be the category of ind-conilpotent cocommutative counital  $\Z$-graded DG coalgebras, denoted by $\mathrm{dgcu}(k)$ in   \cite{hinstack} 2.1.2.
\end{definition}
 Beware that  Hinich uses the somewhat misleading terminology ``unital'' to mean ``ind-conilpotent counital''. Quillen uses the  term ``connected'' for the same concept.

\begin{definition}\label{clqdef}
Define a functor $\C_q: DG_{\Z}\LA \to DG_{\Z}\mathrm{CU}_k$  by $L \mapsto k \oplus C(L[1])$, as in Lemma \ref{dglainfty}. This functor has a left adjoint $\cL_q$. 
 In \cite{hinstack} \S 2.2, these functors are denoted by $\C$ and $\cL$, respectively.
\end{definition}

\begin{lemma}\label{dglamodel}
The category $DG_{\Z}\LA$  has a cofibrantly generated closed model category structure, in which a map $f: L^{\bt}\to M^{\bt}$ is a fibration if it is surjective, and a weak equivalence if $\H^*(f):\H^*(L^{\bt})\to \H^*(M^{\bt})$ is a weak equivalence. 
\end{lemma}
\begin{proof}
Apply \cite{Hirschhorn} Theorem 11.3.2 to the forgetful functor from DGLAs to cochain complexes.
\end{proof}

\begin{lemma}\label{hinmodel}
There is a model structure on $DG_{\Z}\mathrm{CU}_k$ in which $f: C \to D$ is:
\begin{enumerate}
\item a cofibration if the maps $f^n:C^n \to D^n$  are all injective;
\item a weak equivalence if $\cL_q(f)$ is a quasi-isomorphism.
\end{enumerate}

The functor $\cL: DG_{\Z}\mathrm{CU}_k\to DG_{\Z}\LA$ is then a left Quillen equivalence.
\end{lemma}
\begin{proof}
\cite{hinstack} Theorems 3.1 and 3.2.
 \end{proof}

\begin{definition}\label{hinnerve}
Given a DGLA $L$, recall from \cite{hinstack} Definition 8.1.1 that  the simplicial nerve $\Sigma(L): dg\C_k \to \bS$ is defined  by 
$$
\Sigma(L)(A)_n := \mc(L\ten \sA_n)(A),
$$
where $\sA_n$ is  is the algebra of polynomial differential forms on the standard $n$-simplex $\Delta^n$ (denoted $\Omega_n$ in \cite{hinstack}).
\end{definition}

Now, as in \cite{Hovey} \S 5, for any pair $X,Y$ of objects in a model category $\C$, there is a derived function complex $\bR\Map_{\C}(X, Y) \in \bS$, defined up to weak equivalence. If $\C$ is a simplicial model category, with $X$ is cofibrant and $Y$ fibrant, then 
$$
\bR\Map_{\C}(X,Y) \simeq \HHom_{\C}(X,Y).
$$ 
In general model categories, it suffices to take a cofibrant replacement $\tilde{X}$ for $X$
 and a fibrant resolution $\hat{Y}_{\bt}$ for $Y$ in the Reedy category of simplicial diagrams in $\C$, then to set
 $$
 \bR\Map_{\C}(X,Y)_n:= \Hom_{\C}(\tilde{X}, \hat{Y}_n).
 $$ 
Quillen equivalences of model categories induce weak equivalences on derived function complexes (by applying the associated derived functors to $X$ and $Y$). 

\begin{proposition}\label{hinnerve2}
If $A \in dg\C_k$ and $X \in DG_{\Z}\mathrm{CU}_k$, then
$$
\bR \Map_{DG_{\Z}\mathrm{CU}_k}(A^{\vee}, X) \simeq \Sigma(\cL_q(X)) \in \bS.
$$
\end{proposition}
\begin{proof}
This is \cite{hinstack} Proposition 8.1.2. First observe that all objects of  $DG_{\Z}\mathrm{CU}_k$ are cofibrant, and all objects of $DG_{\Z}\mathrm{CU}_k$ fibrant, so $\cL_q, \C_q$ are equivalent to the associated derived functors $\bR\C, \bL\cL$.

The key fact is that $[n] \mapsto L \ten \sA_n$ is a Reedy fibrant simplicial resolution for $L$ in $DG_{\Z}\LA$, so $[n] \mapsto \C_q (\cL_qX \ten \sA_n)$ is a Reedy fibrant simplicial resolution for $\C_q(\cL_q(X))$. By the observation above, $\C_q(\cL_q(X)) \simeq \bR\C_q(\bL\cL_q(X))$,  which in turn is equivalent to $X$, by the Quillen equivalence of Lemma \ref{hinmodel}. 

Therefore $[n] \mapsto \C_q (\cL_qX \ten \sA_n)$ is a Reedy fibrant simplicial resolution for $X$, so
$$
\Sigma(\cL_q(X))(A)_n =\Hom_{ DG_{\Z}\mathrm{CU}_k}(A^{\vee}, \C_q(\cL_q(X)\ten \sA_n)) 
$$
means that $\Sigma(\cL_q(X))(A) \simeq \bR \Map_{DG_{\Z}\mathrm{CU}_k}(A^{\vee}, X)$.
\end{proof}

\subsection{Global derived stacks}

In \cite{hag2} and \cite{lurie},  derived stacks are defined, with a view to modelling (global) derived moduli.

Given a ring $S$, a geometric $D^{-}$-stack (\cite{hag2}) or a derived stack (\cite{lurie}) over $S$ is a functor
$$
F:s\Alg_S \to \bS
$$
on simplicial $S$-algebras, satisfying many additional conditions. A morphism $F \to G$ of geometric $D^{-}$-stacks is a weak equivalence if it induces weak equivalences $F(A) \to G(A)$ in $\bS$ for all $A \in s\Alg_S$.

A sketch of the definition of geometric $D^{-}$-stacks (\cite{hag2} Definition 1.3.3.1) follows. 

For a  simplicial $S$-algebra $R$, the functor 
\begin{eqnarray*}
\bR\underline{\Spec}R: s\Alg_S &\to& \bS\\
A &\mapsto& \bR\HHom_{s\Alg_S}(R,A)
\end{eqnarray*}
is a   quasi-compact $0$-geometric $D^{-}$-stack.

An arbitrary $0$-geometric $D^{-}$-stack is a disjoint union of quasi-compact $0$-geometric $D^{-}$-stacks, where the disjoint union is taken not in the category of presheaves, but in a subcategory of $\infty$-sheaves, relative to a certain model structure.

$n$-geometric $D^{-}$-stacks are then defined inductively by saying that $F$ is  $n$-geometric  if 
\begin{enumerate}
\item
there exists a homotopy-smooth covering  $p:U \to F$ from a  $0$-geometric $D^{-}$-stack $U$, and
\item the diagonal $F \to F\by F$ is representable by $(n-1)$-geometric $D^{-}$-stacks.
\end{enumerate}

If we take $U$ to be quasi-compact at each stage in the definition above, then we obtain the definition of a strongly quasi-compact $n$-geometric $D^{-}$-stack.

Beware that the derived $n$-stacks of \cite{lurie} are defined slightly differently, taking $0$-stacks to be derived analogues of algebraic spaces, rather than disjoint unions of affine schemes. Thus every $n$-geometric $D^-$-stack is a derived $n$-stack in Lurie's sense. Conversely, every  derived $n$-stack in Lurie's sense is $(n+2)$-geometric.

\begin{definition}\label{toenstalk}
Given an $n$-geometric $D^{-}$-stack $F$ over $S$, take a point $x: \Spec k \to F$ for a field $k$, such that the composition $s: \Spec k \to \Spec S$ is a closed point. Let $\L$ be the formal completion of $S$ at $s$, and define the formal neighbourhood 
$$
F_x: s\hat{\C}_{\L} \to \bS
$$
by
$$
F_x(A):= F(A)\by^h_{F(k)}\{x\}.
$$
\end{definition}

\begin{proposition}\label{toencomp}
A formal neighbourhood $F_x$ of an $n$-geometric $D^-$-stack $F$ at a point $x$  is representable by an object of $\Ho(sc\Sp)$.
\end{proposition}
\begin{proof} First observe that Corollary \ref{weakchar} ensures that the notions of weak equivalence for $D^-$-stacks and $sc\Sp$ are compatible. By Theorem \ref{schrep}, it will  suffice to show that $F_x \in \cS$. $D^-$-stacks  automatically preserve weak equivalences, so $F_x$ satisfies (A2). 

It therefore  suffices to prove (A1): that for any square-zero extension $A \to C$ and any morphism $B \to C$ 
in $s\Alg_S$, the map
$$
F(A\by_CB) \to F(A)\by^h_{F(C)}F(B)
$$
is a weak equivalence in $\bS$. This is very similar to \cite{lurie} Proposition 5.3.7, which proves this for the case $B=A$. Adapting the proof of that proposition, it suffices to show that for any homotopy-smooth surjective map  $U \to X$ of $n$-geometric $D^-$-stacks, the map
$$
U(A)\by_{U(C)}^h U(B) \to X(A)\by_{X(C)}^h X(B)
$$ 
is surjective. Moreover, the argument of \cite{lurie} Proposition 5.3.7 allows us to replace $A\by_CB$ with a homotopy  \'etale algebra over it, giving a local lift of a point $x \in X(B)$ to $u \in U(B)$. The problem then reduces to showing that
$$
U(A)\by_{U(C)}^h U(B) \to X(A)\by_{X(C)}^h U(B)
$$ 
is surjective, but this follows from pulling back the surjection
$$
U(A) \to U(C)\by_{X(C)}^hX(B)
$$ 
given by the homotopy-smoothness of $U \to X$.
\end{proof}

\begin{remark}
Corollary \ref{toencomp} has a partial converse, in the sense that a quasi-smooth $X \in sc\Sp$ with  $\H^i(X)=0$ for $i<-n$ satisfies the formal criteria for representability by a derived $n$-stack, namely \cite{lurie} Theorem 7.5.1 (2),(3),(4),(5): Boundedness of $\H^*(X)$ implies (2) ($n$-truncation); (A2) implies (3) (\'etale sheaf), since a map in $s\C_{\L}$ is \'etale in the sense of [ibid.] only if it is a weak equivalence; (A1) implies (4) (cohesiveness); (5) (nilcompleteness) follows from our formula for extending  $F$ from $s\C_{\L}$ to $s\hat{\C}_{\L}$. Of the other conditions, (1) and (7) are concerned with finiteness, while (6) is a global property, describing effectiveness of formal deformations.  
\end{remark}

\begin{remark}
In fact, there is now a global version of Proposition \ref{toencomp}. Every geometric $D^{-}$-stack can be represented by a   simplicial complex of disjoint unions of cosimplicial affine schemes (\cite{stacks2} Theorem \ref{stacks-relstrict}). Moreover, every  strongly quasi-compact geometric $D^{-}$-stack can be represented by a cosimplicial simplicial affine scheme.  
\end{remark}

\begin{remark}\label{dgtoen}
If $S$ is of characteristic $0$, then there is an alternative, equivalent, formulation of $n$-geometric $D^{-}$-stacks as functors $F:dg\Alg_S \to \bS$ on (non-negatively graded) chain algebras. That this is equivalent makes use of the Quillen equivalence between $dg\Alg_S$ and $s\Alg_S$ from \cite{QRat}. The proof runs along the same lines as Theorem \ref{nequiv}.
\end{remark}

\section{Comparison with SHLAs}\label{alternative}

From now, on assume that the residue   field $k$ is  of characteristic $0$.

\subsection{Pro-Artinian chain algebras}
\begin{definition}
  Let  $dg\hat{\C}_{\L}$ be the category of pro-objects of $dg\C_{\L}$. Write $DG\Sp:= (dg\hat{\C}_{\L})^{\op}$; this is equivalent to the category of left-exact set-valued functors on  $dg\C_{\L}$. Given $R \in dg\hat{\C}_{\L}$, let $\Spf R \in DG\Sp $ denote the corresponding object in the opposite category. 
\end{definition}

\begin{definition}\label{dgmcldef} In the category $dg\hat{\C}_{\L}$, we say that $R\to S$ is:
\begin{enumerate}
\item
 a fibration if $R_i \to S_i$ is surjective for all $i>0$;
\item
a weak equivalence if it is   acyclic (i.e. a quasi-isomorphism);  
\item  
a cofibration if it has the LLP with respect to all acyclic fibrations; these maps are also called quasi-smooth.
\end{enumerate}
\end{definition}

Observe that every surjection $A \onto B$ in $dg\hat{\C}_{\L}$ is a fibration.

\begin{proposition}\label{dgmcl}
With the classes of morphisms  given above, $dg\hat{\C}_{\L}$ is a model category.
\end{proposition}
\begin{proof}
As for Proposition \ref{smcl}.
\end{proof}

\begin{definition}
Define a map $A \to B$ in $dg\C_{\L}$ to be a small extension if it is surjective and the kernel $I$ satisfies $I\cdot \m(A)=0$. 
\end{definition}

\begin{lemma}
Every surjection in  $dg\C_{\L}$ can be factored as a composition of small extensions, and every acyclic surjection as a composition of acyclic small extensions.
\end{lemma}
\begin{proof}
The proof of Lemma \ref{small} carries over to this context.
\end{proof}

\subsection{Cosimplicial pro-Artinian chain algebras}

\begin{definition}
Define $cdg\hat{\C}_{\L}:=(dg\hat{\C}_{\L})^{\Delta}$ to be the category of cosimplicial pro-Artinian chain algebras. Let $sDG\Sp:=(cdg\hat{\C}_{\L})^{\op}$ be the opposite category, or equivalently the category of left-exact functors from $dg\C_{\L}$ to $\bS$.
\end{definition}

\begin{remark}\label{hinpshf}
If $\L=k$, note that this category is a subcategory of the category of simplicial presheaves on $dg\C_{\L}$ used in \cite{hinstack} \S 8 to model nerves of DGLAs.
\end{remark}

\begin{definition}
Given $X \in sDG\Sp, K \in \bS$, define $X^K$ by $X^K(A):= X(A)^K \in \bS$, for $A \in dg\C_{\L}$.
\end{definition}

\begin{definition}
Say a map $X \to Y$ in $sDG\Sp$ is quasi-smooth if it maps small extensions in $dg\C_{\L}$ to fibrations in $\bS$, and acyclic small extensions to trivial fibrations. 
\end{definition}

\begin{definition}
Given a quasi-smooth map $E \xra{p} B$ in $sDG\Sp$, and a morphism $X\to B$ in  $sDG\Sp$, define $[X,E]_B$ to be the coequaliser
$$
\xymatrix@1{\Hom_{sDG\Sp\da B}(X,E^{\Delta^1}\by_{B^{\Delta^1}}B) \ar@<.5ex>[r]\ar@<-.5ex>[r] &\Hom_{sDG\Sp\da B}(X,E) \ar[r] &[X,E]_B},
$$ 
similarly to Definition \ref{Xpdef}.
\end{definition}

\begin{definition}\label{dgdefcatdef}
Given  a map $f:X \to Y$ in the category $sDG\Sp$, with $X=\Spf S, Y=\Spf R$, say that $f$ is:  
\begin{enumerate}
\item
 a geometric cofibration if $(f^{\sharp})^n_i: R^n_i \to S^n_i$ is surjective for all $i,n\ge 0$;
\item a geometric weak equivalence if for all quasi-smooth maps $p:E\to Y$, 
$$
f^*:[Y,E]_Y \to [X,E]_Y
$$
is an isomorphism;

\item  
a geometric fibration if $f$ is quasi-smooth.
\end{enumerate}
\end{definition}

\begin{proposition}\label{sdgmodel}
The category $sDG\Sp$ is a simplicial model category with the geometric model structure. 
\end{proposition}
\begin{proof}
The proof of  Theorem \ref{scspmodel} carries over to this context.
\end{proof}

\begin{lemma}\label{dgweakext}
Take a surjection $f:A \to B$ in $cdg\hat{\C}_{\L}$ with kernel $I$, such that $\m(A)\cdot I=0$. Then $f$ is a weak equivalence if and only if $\H_*(\Tot^{\Pi} N_cI)=0$.
\end{lemma}
\begin{proof}
The proof of  Lemma \ref{weakext} carries over to this context.
\end{proof}

\begin{theorem}\label{dgschrep}
There is a canonical equivalence between the geometric homotopy category $\Ho(sDG\Sp)$ and the homotopy category $\Ho(\cS')$ of functors $F: dg\C_{\L} \to \bS$ satisfying the analogues for $dg\C_{\L}$ of conditions (A0)--(A2) from Definition \ref{schless}.
\end{theorem}
\begin{proof}
The proof of   Theorem \ref{schrep} carries over to this context.
\end{proof}

\subsection{Normalisation}

\begin{definition}\label{qrat}
Define the normalisation functor $N: s\C_{\L} \to dg\C_{\L}$ by mapping $A$ to its associated normalised complex $NA$, equipped with the Eilenberg-Zilber shuffle product (as in \cite{QRat}).
\end{definition}

\begin{lemma}\label{nworks}
$N: s\hat{\C}_{\L} \to dg\hat{\C}_{\L}$ is a right Quillen equivalence.
\end{lemma}
\begin{proof}
It is immediate from the definitions that $N$ is a right Quillen functor, as it preserves limits, takes fibrations to fibrations, and takes weak equivalences to weak equivalences. The argument of \cite{QRat} Theorem I.4.6 shows that the unit $R \to NN^*R$ of the adjunction is a weak equivalence for all cofibrant $R \in dg\hat{\C}_{\L}$. Given an arbitrary element $A \in s\hat{\C}_{\L}$, we need to show that the co-unit $\vareps:N^*\widehat{NA} \to A$ is a weak equivalence, for a cofibrant approximation $\widehat{NA}$ of $NA$. But $\widehat{NA} \to NN^*\widehat{NA} $ is a weak equivalence, so $NN^*\widehat{NA} \to NA$ must be, and hence $\vareps$ is, as $N$ reflects isomorphisms.
\end{proof}

\begin{definition}
Define $\Spf N^*:  sDG\Sp \to sc\Sp$ by mapping $X: dg\C_{\L} \to \bS$ to the composition $X\circ N:s\C_{\L} \to \bS$. Note that this is well-defined, since $N$ is left-exact.
\end{definition}

\begin{theorem}\label{nequiv}
$\Spf N^*:  sDG\Sp \to sc\Sp$ is a right Quillen equivalence.
\end{theorem}
\begin{proof}
$\Spf N^*$ is clearly continuous, so it is a right adjoint. To see that it is a right Quillen functor, just observe that $N$ sends surjections to surjections, and acyclic surjections to acyclic surjections. In order to see that this is a right Quillen equivalence, it suffices to show that the derived functor $\bR\Spf N^*:  \Ho(sDG\Sp) \to \Ho(sc\Sp)$ is an equivalence.

We now observe that Theorems \ref{schrep} and \ref{dgschrep} show that  $\Ho(sc\Sp)$ (resp. $\Ho(sDG\Sp)$) is  equivalent to the homotopy category consisting of those functors from $s\C_{\L}$ (resp. $dg\C_{\L}$) to $\Ho(\bS)$ with $F(k)\simeq\bt$ and for which   
$ 
F(A\by_BC) \to F(A)\by_{F(B)}^h F(C)
$
is a weak equivalence  in $\bS$ whenever $\pi_0A \to \pi_0B$ is surjective. Under these equivalences, $\bR\Spf N^*:  \Ho(sDG\Sp) \to \Ho(sc\Sp)$ corresponds to the functor $ N^*: \cS' \to \cS$ given by $N^* F(A):= F(NA)$.  

Now, the normalisation functor $N: s\C_{\L}\to dg\C_{\L}$ is a right Quillen equivalence, by Lemma \ref{nworks}; denote the derived left adjoint by $\bL N^*$. We may then define a functor $ (\bL N^*)^*:\cS \to \cS'$ by $ (\bL N^*)^*F(A):=F(\bL N^*A)$. This is well-defined because these functors preserve homotopy groups and homotopy limits. Since the functors $N$ and $\bL N^*$ are homotopy inverses, they induce equivalences $\Ho(\cS) \simeq \Ho(\cS')$.

This shows that $\bR\Spf N^*$ yields an equivalence of homotopy categories, as required. 
\end{proof}

\subsection{Pro-Artinian cochain chain algebras and denormalisation}

\begin{definition}
Define $DG dg\C_{\L}$ to be the category of Artinian local  $\N_0\by\N_0$-graded  graded-commutative  $\L$-algebras $A^{\bt}_{\bt}$ with differential of bidegree $(1,-1)$ and  residue field $k$. Let  $DGdg\hat{\C}_{\L}$ be the category of pro-objects of $DGdg\C_{\L}$, and denote its opposite category by $dgDG\Sp$. 
\end{definition}

\begin{definition}\label{nabla}
Define the denormalisation functor $D:DG dg\hat{\C}_{\L}\to cdg\hat{\C}_{\L}$ by
$
D(A):= (N_c^{-1})(A),
$
for $N_c$ the normalisation functor for cochain complexes. The multiplication on $DA$ is then defined using the Eilenberg-Zilber shuffle product $\nabla: D^m(A)\by D^m(A) \to D^m(A)$. 

Explicitly, we first form the denormalised cosimplicial complex as the formal sum 
$$
D^nA:= \bigoplus_{\begin{smallmatrix} m+s=n \\ 1 \le j_1 < \ldots < j_s \le n \end{smallmatrix}} \pd^{j_s}\ldots\pd^{j_1}A^m.
$$
 We then  define the operations $\pd^j$ and $\sigma^i$ using the cosimplicial identities, subject to the conditions that $\sigma^i A =0$ and $\pd^0a= da -\sum_{i=1}^{n+1}(-1)^i \pd^i a$ for all $a \in A^n$.

We now have to define the product $\nabla$ from $D^nA \ten D^nA$ to $D^n A$. Given a finite set  $I$ of strictly positive integers, write $\pd^I= \pd^{i_s}\ldots\pd^{i_1}$, for $I=\{i_1, \ldots i_s\}$, with $1 \le i_1 < \ldots < i_s$. The product is then   defined on the basis by 
$$
(\pd^Ia)\nabla (\pd^J b):= \left\{ \begin{matrix} \pd^{I\cap J}(-1)^{(J\backslash I, I \backslash J)} (a\cdot b) & |a|= |J\backslash I|, |b|= |I\backslash J|,\\ 0 & \text{ otherwise},\end{matrix} \right.
$$
where for disjoint sets $S,T$ of integers, $(-1)^{(S,T)}$ is the sign of the shuffle permutation of $S \sqcup T $ which sends the first $|S|$ elements to $S$ (in order), and the remaining $|T|$ elements to $T$ (in order). 

Beware that this description only works for $0 \notin I \cup J$.

Observe that $D$ is continuous, so has  left adjoint $D^*$.
\end{definition}

\begin{definition}\label{dgdgdefcatdef} Given  a map $f:R \to S$ in the category $ DGdg\hat{\C}_{\L} $, say that $f$ is:  
\begin{enumerate}
\item
 a geometric fibration if $Df$ is a geometric fibration in $cdg\hat{\C}_{\L}$;
\item a geometric weak equivalence if  $Df$ is a geometric weak equivalence in $cdg\hat{\C}_{\L}$;
\item  
a geometric cofibration if it has the left lifting property with respect to all trivial fibrations.
\end{enumerate}
\end{definition}

\begin{definition}
Define a surjective map $f: A \to B$ in $DGdg\C_{\L}$ to be a small extension if it is surjective   with kernel $V$, such that $\m(A)\cdot V=0$.
Define $P$ to be the class of small extensions in $DGdg\C_{\L}$, and $Q\subset P$ to consist of those small extensions for which $\H_*(\Tot^{\Pi}V)=0$.  \end{definition}

\begin{lemma}\label{Dsmallfull} 
Given $A \in DGdg\C_{\L}$, every small extension $DA \to B$ is isomorphic to $Df$, for some small extension $f:A\to C$ in $cdg\C_{\L}$.
\end{lemma}
\begin{proof}
Take an ideal $I \lhd DA$ with $I \nabla D\m(A)=0$ (for $\nabla$ as in Definition \ref{nabla}). It suffices  to show that $(NI) \cdot \m(A)=0$, since this forces $NI$ to be an ideal, and we may set $C=A/NI$. Now, observe that given $x \in N^mI, a \in \m(A^n)$, we have $x\cdot a= ((\pd^{m+1})^nx )\nabla ((\pd^0)^m a)=0 \in A^{m+n}$, as required.
\end{proof} 

\begin{corollary}\label{dfibrnworks}
Given $A \in DGdg\hat{\C}_{\L}$, every  fibration  $DA \to B$ lies in the essential image of $D$. 
\end{corollary}

\begin{lemma}\label{dunit}
Given a cofibration $j:R \to S$ in $cdg\hat{\C}_{\L}$, an object $T \in DGdg\hat{\C}_{\L}$, and a morphism $ R \to DT$, the canonical map
$$
 f :S\ten_RDT \to D(D^*S\ten_{D^*R}T),
$$
where $\ten$ on the right-hand side denotes  graded tensor product,  is a trivial fibration.
\end{lemma}
\begin{proof}
Take the filtration $F^i (S\ten_RDT)= DT\m(S)^i + \m(DT)\m(S)^{i-2}$  (similarly to \S \ref{adams}), and observe that on the associated graded pieces, we have 
\begin{eqnarray*}
&\Gr^if:\Symm^i \cot (S/R) \oplus \bigoplus_{r=1}^{i-1} (\m(DT)^r/\m(DT)^{r+1})\ten \Symm^{i-1-r}\cot (S/R)\\
 &\to N^{-1}\Symm^i N\cot(S/R)  \oplus \bigoplus_{r=1}^{i-1} (\m(DT)^r/\m(DT)^{r+1})\ten N^{-1}\Symm^{i-1-r}N\cot (S/R),
\end{eqnarray*}
where the tensor product and symmetric functor on the right-hand side  follow the usual graded conventions. These maps are all surjective, so $f$ must be a fibration. 

Note that $\Gr^if$ is also a quasi-isomorphism, in the sense that  $H_*(\Tot^{\Pi}N\ker(\Gr^if))=0$. Now,  $\ker(\Gr^if)$ is the kernel of the small extension
$$
f_i:(S\ten_RDT)/F^{i+1} \to (D(D^*S\ten_{D^*R}T)/F^{i+1})\by_{D(D^*S\ten_{D^*R}T)/F^i}(S\ten_RDT)/F^i,
$$
which is  a trivial fibration by Lemma \ref{weakext}.
 Thus $f$ is a transfinite composition of pullbacks of trivial fibrations, so must be a trivial fibration.
\end{proof}

\begin{theorem}\label{dequiv}
With the structures above, $DGdg\hat{\C}_{\L}$ is a closed model category. It is fibrantly cogenerated, with cogenerating fibrations $P$ and cogenerating trivial fibrations $Q$. Moreover, $D:DG dg\hat{\C}_{\L}\to cdg\hat{\C}_{\L}$ is a right Quillen equivalence.
\end{theorem}
\begin{proof}
From Corollary \ref{dfibrnworks} and Proposition \ref{sdgmodel}, we know that fibrations and trivial fibrations are relative $P$-cells and relative $Q$-cells, respectively.

We may now apply \cite{Hovey} Theorem 2.1.19 to show we have a closed model category structure. The only non-trivial condition to verify is that the class of $P$-projectives is the intersection of the classes of weak equivalences and of $Q$-projectives.

Since $Q \subset P$, every $P$-projective is $Q$-projective. Given a $Q$-projective $f:R \to S$, take factorisations $DR \xra{i} \widetilde{DS} \xra{p} DS$, $DR \xra{i'} \widetilde{DS}' \xra{p'} DS$   of $Df$ in $cdg\hat{\C}_{\L}$, with $i$ a cofibration, $i'$ a trivial cofibration, $p$ a trivial fibration and $p'$ a  fibration. The adjoint maps $D^*\widetilde{DS} \to S$,  $D^*\widetilde{DS}' \to S$ to $p,p'$ are clearly surjective, as are $q:R\ten_{D^*DR}D^*\widetilde{DS}\to S$, $q':R\ten_{D^*DR}D^*\widetilde{DS}'\to S$.

Observe that by Lemma \ref{dunit},
$$
\widetilde{DS} \to D(R\ten_{D^*DR}D^*\widetilde{DS})
$$
is a weak equivalence, so $Dq$ must be a trivial fibration, hence $q$ is a relative $Q$-cocell. Since $f$ is $Q$-projective, we may therefore choose a section $s$ of $q$ over $R$. 

If $f$ is a weak equivalence, then  $i$ is a trivial cofibration, so $D^*i$ is a $P$-projective, as is $f':R \to R\ten_{D^*DR}D^*\widetilde{DS}$. Since $f$ is a retraction of  $f'$, it must also be a $P$-projective.

Conversely, if $f$ is a $P$-projective, then $q'$ has a section over $R$. Therefore $q'$ is a retraction of $D^*i': D^*DR \to D^*\widetilde{DS}' $. By Lemma \ref{dunit},
 $$
\widetilde{DS}'\ten_{DR}DD^*R \to DD^*\widetilde{DS}'
$$
is a weak equivalence, so $DD^*i'$ (and hence $D^*i'$) must also be (as $i':DR \to \widetilde{DS}'$ is a weak equivalence, so the left-hand side is weakly equivalent to $DD^*R$). Thus $q'$ is a weak equivalence.

We have now established that $DGdg\hat{\C}_{\L}$ is a closed model category, and that $D$ is a right Quillen functor. It remains only to show that $D$ is a right Quillen equivalence. Given $R \in cdg\hat{\C}_{\L}$ cofibrant, Lemma \ref{dunit} implies that $\eta: R \to DD^*R$ is a weak equivalence. Given $S \in  DGdg\hat{\C}_{\L}$, take a cofibrant approximation $q:\widetilde{DS}\to DS$, and consider $\vareps:D^*\widetilde{DS}\to S$. We know that $\widetilde{DS}\to DD^*\widetilde{DS}$ is a weak equivalence, as is $q$, so $D\vareps$ (and hence $\vareps$) must also be a weak equivalence.  
\end{proof}

\begin{remark}\label{ntanfdef2}
Observe that under this correspondence, the Eilenberg-Maclane spaces $K(n)$ of Definition \ref{eilmac} correspond (up to weak equivalence) to the objects $k\oplus k^{[i-n]}_{[i]}\eps \in DGdg\hat{\C}_{\L}$, where $k^{[j]}_{[i]}$ is the bicomplex with $k$ concentrated in degree $(-j,-i)$.

Moreover, observe that, for $R \to S$ cofibrant in $DGdg\hat{\C}_{\L}$, the cotangent complex $\cot(S/R):= \m(S)/(\m(R)+\m(S)^2)$ is quasi-smooth in the sense of Lemma \ref{decompcot}. If we write $t(S/R):=\cot(S/R)^{\vee}$ (making use of Lemma \ref{profd}), then 
$$
\H^*( \Spf D^*S/ \Spf D^*R)\cong \H^*(\Tot t(S/R)),
$$
and this can detect weak equivalences between cofibrant objects.
\end{remark}

\subsection{$\Z$-graded pro-Artinian chain algebras}

We have now reached the stage where we may compare our categories with those arising in \cite{Kon}, \cite{Man2} and \cite{hinstack}. The main comparison will be with the category $DG_{\Z}\Sp$ of Definition \ref{dgzsp}, which will be given a model structure in this section.

\begin{remarks}\label{cfothers}
 Remark \ref{manchar} and Theorem \ref{allequiv}  will imply that objects of  $\Ho(DG_{\Z}\Sp)$ correspond to the geometric 
deformation functors of \cite[\S 7]{Man2}.
 Corollary \ref{manho} will then  show that for any $L_{\infty}$-algebra $V$, $\ddef(V)$ is just the functor on $ \Ho(DG_{\Z}\Sp)$ represented by $\mc(V) \in DG_{\Z}\Sp$.

Dualising an object $A \in dg_{\Z}\hat{\C}_{\L}$ gives a DG coalgebra $A^{\vee}$ (as in Remark \ref{profd2}). If $\L=k$, this allows us to regard $DG_{\Z}\Sp$ as a subcategory of the category $DG\mathrm{CU}(k)$ of DG coalgebras considered in \cite{hinstack}.  Not all DG coalgebras arise in this way, only those which are ind-conilpotent (i.e. unions of conilpotent coalgebras). However,  Corollary \ref{cfhin} will  show that the model categories $DG_{\Z}\Sp$ and  $DG\mathrm{CU}(k)$ are Quillen-equivalent. Proposition \ref{hinnerveprop} will then show that this equivalence is given by Hinich's simplicial nerve functor.

Given an SHLA $C$, the dual $C^{\vee}$ is an object of $dg_{\Z}\hat{\C}_{k}$, and Proposition \ref{tqiscor} will show that this gives an equivalence between $\Ho(DG_{\Z}\Sp)$ and the homotopy category of SHLAs considered in \cite{Kon}.
\end{remarks}

\begin{definition}
Define $P$ to be the class of small extensions in  $dg_{\Z}\C_{\L}$, and $Q\subset P$ to consist of those small extensions for which $\H_*(V)=0$.  
\end{definition}

\begin{remark}
Every surjection in $dg_{\Z}\hat{\C}_{\L}$ is a relative $P$-cocell, but not every acyclic surjection is a relative $Q$-cocell. 
\end{remark}

\begin{definition}\label{dgpathcyl}
Given $A \in dg_{\Z}\C_{\L}$, form the free chain algebra $A[t,dt]$ over $A$, for $t$ of degree $0$. For $i=0,1$, define $\ev_i:A[t,dt]\to A$ by mapping $t$ to $i$, and consider the chain algebra 
$$
D:=A[t,dt]\by_{k[t,dt]}k.
$$ 
Define the path object $\cP A \in dg_{\Z}\hat{\C}_{\L}$ to be the completion of $D$ with respect to the augmentation ideal of the map $(\ev_0,\ev_1):D \to A\by_kA$. Note that there is a canonical map $A \to \cP A$ which is a section of both $\ev_0$ and $\ev_1$.

Observe that the functor $\cP :dg_{\Z}\C_{\L}\to  dg_{\Z}\hat{\C}_{\L}$ is left-exact, and extend it to $dg_{\Z}\hat{\C}_{\L}$ by continuity. Given $R \in  dg_{\Z}\hat{\C}_{\L}$, define the cylinder object $CR$ to pro-represent the functor $A \mapsto \Hom(R, \cP A)$ (noting that this is left-exact).
\end{definition}

\begin{lemma}\label{dgpathworks}
For $A\in dg_{\Z}\hat{\C}_{\L}$, the maps $\ev_i: \cP A \to A$ are relative $Q$-cocells for $i=0,1$, and the map $(\ev_0,\ev_1):\cP A \to A\by_kA$ is a relative $P$-cocell.

For any relative $P$-cocell $A\to B$, the maps $\ev_i:\cP A\by_{\cP B}B \to A$  are relative $Q$-cocells, and  $(\ev_0,\ev_1):\cP A\by_{\cP B}B \to A\by_BA$ is a relative $P$-cocell.

 If $A\to B$ is a relative $Q$-cocell, then so is $(\ev_0,\ev_1):\cP A\by_{\cP B}B \to A\by_BA$.

\end{lemma}
\begin{proof}
We prove the first statement; the second is similar.
It is immediate that $(\ev_0,\ev_1):\cP A \to A\by_kA$ is surjective, hence a relative $P$-cocell.

Write $J$ for the kernel of $(\ev_0,\ev_1):B\to A\by_kA$, and observe that the ideal $J^n+tJ^{n-1}= t^{n-1}(t-1)^{n-1}(t,dt)$. Thus the
quotients $\cP _n:=D/(J^n+tJ^{n-1})$ have the property that $\cP _{n+1} \to \cP _n$ is a relative $Q$-cocell, factorising as the acyclic small extensions $\cP _{n+1}\to D/t^n(t-1)^{n-1}(t,dt)\to \cP _n$. Since the systems $\{J^n+tJ^{n-1}\}$ and $\{J^n\}$ of ideals  define the same topology, and $\cP _1=A$, this means that $\ev_0$ is a relative $Q$-cocell, as is $\ev_1$, by symmetry.

For the final statement, it suffices to consider the case when $A \to B$ is in $Q$, with kernel $I$. Then $\cP A\by_{\cP B}B \to A\by_BA$ has kernel $t(t-1)I$, and the system $t^{n}(t-1)^{n}I \to t^{n}(t-1)^{n-1}I \to t^{n-1}(t-1)^{n-1}I$ of ideals gives rise to a sequence of acyclic small extensions, as required.
\end{proof}

\begin{corollary}
If $f:R\to S$ in $dg_{\Z}\hat{\C}_{\L}$ is $Q$-projective, then there are $P$-projective maps $\iota_0,\iota_1:S \to CS\ten_{CR}R$, with $\iota_0\ten \iota_1:S\ten_R S \to CS\ten_{CR}R$ $Q$-projective. If $f$ is moreover $P$-projective, then so is $\iota_0\ten \iota_1$.
\end{corollary}
\begin{proof}
Apply the description $\Hom(CR,A)=\Hom(R,\cP A)$ to Lemma \ref{dgpathworks}.
\end{proof}

\begin{definition}
Say that a map $p:X \to Y$ in $DG_{\Z}\Sp$ is quasi-smooth (resp. trivially quasi-smooth) if it is dual to a $Q$-projective (resp. a $P$-projective).
\end{definition}

\begin{definition}\label{dgzweakdef}
Given $X \in DG_{\Z}\Sp$, given by $\Spf R$ for $R \in dg_{\Z}\hat{\C}_{\L}$  set $X^I:=\Spf(CR)$. Given a quasi-smooth map $p:X \to Y$ in $DG_{\Z}\Sp$,
and a morphism $U\to Y$ in $sc\Sp$, define $[U,X]_Y$ to be the coequaliser
$$
\xymatrix@1{\Hom_{DG_{\Z}\Sp\da Y}(U,X^I\by_{Y^I}Y) \ar@<.5ex>[r]\ar@<-.5ex>[r] &\Hom_{DG_{\Z}\Sp\da Y}(U,X) \ar[r] &[U,X]_Y},
$$ 
similarly to Definition \ref{Xpdef}.

Say that a map $f:U \to V$ in $DG_{\Z}\Sp$ is 
a weak equivalence if for all quasi-smooth maps $p:X\to V$, 
$$
f^*:[V,X]_V \to [U,X]_V
$$
is an isomorphism.
\end{definition}

\begin{proposition}\label{dgspmodel}
There is a cofibrantly generated closed model structure on $dg_{\Z}\hat{\C}_{\L}$ with cogenerating fibrations $P$ and cogenerating trivial fibrations $Q$. Weak equivalences are as in Definition \ref{dgzweakdef}.
\end{proposition}
\begin{proof}
The proof of Theorem \ref{scspmodel} carries over.
\end{proof}

\begin{definition}
Given $X \in DG_{\Z}\Sp$ and $A \in dg_{\Z}\hat{\C}_{\L}$, write
$$
X[A]:= [\Spf A,  X] =\Hom_{\Ho(DG_{\Z}\Sp)}(\Spf A,  X).
$$
\end{definition}

\begin{definition}
Given $V \in dg_{\Z}\widehat{\FD\Vect}_k$ and $X \in DG_{\Z}\Sp$, define
$$
\H^n(X\hat{\ten} V):= X[k \oplus V[-n]\eps].
$$
Let $\H^n(X):= \H^n(X \ten k)$, and observe that 
$$
\H^n(X\hat{\ten} V)\cong \prod_{i \in \Z} \H^{n+i}(X)\hat{\ten} \H_i(V),
$$ 
where for  $U\in \Vect_k$ and $W = \{W_{\alpha}\}_{\alpha\in I} \in \widehat{\FD\Vect}_k$, we write $U\hat{\ten}W:= \Lim_I U\ten W_{\alpha}$.
\end{definition}

\begin{lemma}\label{dgcohocalc}
If $X \in DG_{\Z}\Sp$ is quasi-smooth, then $\H^n(X\ten V)$ can be calculated as the quotient space
$$
X(k \oplus V[-n]\eps)/X(k \oplus (V[-n]\ten L^0)  \eps),
$$
for $L^0$ as in Definition \ref{knlndef}.
\end{lemma}
\begin{proof}
$
k \oplus (V[-n]\ten L^0\oplus V[-n] )  \eps
$ 
is a path object for $k \oplus V[-n]\eps$ in $dg_{\Z}\hat{\C}_{\L}$.
\end{proof}

\begin{proposition}\label{dgobs}
If $X \in DG_{\Z}\Sp $, then for any small extension $I \xra{e} A \xra{f} B$, there is a sequence of sets 
$$
X[A]\xra{f_*} X[B] \xra{o_e} \H^1(X\hat{\ten} I), 
$$  
exact in the sense that the fibre of $o_e$ over $0$ is the image of $f_*$. Moreover,  there is a group action of $\H^0(X \hat{\ten} I)$ on $X[A]$ whose orbits are precisely the fibres of $f_*$. 
\end{proposition}
\begin{proof}
This is similar to Theorem \ref{robs}. Let $C(A,I):= (A \oplus (I\ten L^0\eps))/(e +\eps)I$ be the mapping cone of $e$, where $\eps^2=0$. Then $C(A,I) \xra{(f,0)} B$ is a small acyclic surjection, so $X[C(A,I)] \to X[B]$ is an isomorphism. 

Now,
$$
A = C(A,I)\by_{k \oplus I[-1]\eps}k,
$$ 
and since $C(A,I)\to k \oplus I[-1]\eps$ is a fibration, $A$ is the homotopy fibre product, and
$$
X[A] \to X[C(A,I)]\by_{\H^1(X \ten I)}\{0\}
$$
is surjective. This proves the first part.

For the second, note that $A\by_BA \cong A\by_k(k \oplus I\eps)$, so
$$
X[A] \by \H^0(X \ten I) =X[A\by_k(k \oplus I\eps)]\cong  X[A\by_BA] \onto X[A]\by_{X[B]}X[A]. 
$$
\end{proof}

\begin{corollary}\label{dgcohoweak}
A map $f: X \to Y$ in $DG_{\Z}\Sp $ is a weak equivalence if and only if $f_*:\H^*(X) \to \H^*(Y)$ is an isomorphism. 
\end{corollary}

\begin{proposition}\label{tqiscor}
If $\L=k$, then the category $\Ho(DG_{\Z}\Sp)$ is equivalent to the  category of SHLAs localised at tangent quasi-isomorphisms. 
\end{proposition}
\begin{proof}
Observe that an object $R_{\bt} \in dg_{\Z}\hat{\C}_k$ is quasi-smooth precisely when the underlying graded pro-algebra $R_*$ is of the form $R_*= k[[V_*]]$, for some graded pro-finite-dimensional vector space $V_*$. As in Remark \ref{profd2}  the dual  $\m(R_{\bt})^{\vee}$  is a dg coalgebra, and 
$$
\m(R_*)^{\vee} \cong \bigoplus_{n>0} \Gamma^n(V_*)^{\vee},
$$ 
so $\m(R)^{\vee}$ is an SHLA (as in Definition \ref{linfty}). Conversely, given an SHLA $C$, the object $k \oplus C^{\vee} \in dg_{\Z}\hat{\C}_k$ is quasi-smooth. Therefore the functor
$$
C \mapsto \Spf(k \oplus C^{\vee})
$$
gives an equivalence between the category of SHLAs and the full subcategory of quasi-smooth (i.e. fibrant) objects in $DG_{\Z}\Sp$.

It therefore remains only to show that the morphisms $f: C \to D$ of SHLAs which become weak equivalences in $DG_{\Z}\Sp$ are precisely the tangent quasi-isomorphisms.

Set $\cot(S):= \m(S)/(\m(S)^2)$ and $t(S):= \cot(S)^{\vee}$. If $S \in dg_{\Z}\hat{\C}_k$ is quasi-smooth, then Lemma \ref{dgcohocalc} implies that $\H^*(\Spf S)\cong \H^*(t(S))$. Therefore for an SHLA $C$, 
$$
\H^*(\Spf(k \oplus C^{\vee})) \cong \H^*(\tan(C)).
$$
Corollary \ref{dgcohoweak} then implies that $\Spf(k \oplus C^{\vee}) \to \Spf(k \oplus D^{\vee})$ is a weak equivalence if and only if $\tan(C) \to \tan(D)$ is a quasi-isomorphism, as required.
\end{proof}

\begin{definition}
A functor $F:dg_{\Z}\hat{\C}_{\L}\to \Set$ is said to be homotopy representable if there is an object $X \in DG_{\Z}\Sp$ and a natural isomorphism
$$
F(A)\cong X[A].
$$
 \end{definition}

\begin{lemma}\label{manrep}
A  functor $F:dg_{\Z}\hat{\C}_{\L}\to \Set$ is homotopy representable if and only if it satisfies the following conditions:
\begin{enumerate}
\item For any (possibly empty) set of of objects $A_i \in dg_{\Z}\hat{\C}_{\L} $, the map $F(\prod_iA_i) \to \prod_i F(A_i)$    is an isomorphism.

\item For all surjections $A \onto B$ and morphisms $C \to B$ in $dg_{\Z}\hat{\C}_{\L} $, the  map 
$$
F(A\by_BC)\to F(A)\by_{F(B)}F(C)
$$ is surjective. 

\item For all trivial fibrations $A \onto B$ in $dg_{\Z}\hat{\C}_{\L}$, the map $F(A) \to F(B)$ is an isomorphism.
\end{enumerate}
\end{lemma}
\begin{proof}
The final condition  ensures that $F$ this descends to a functor $F:\Ho(dg_{\Z}\hat{\C}_{\L})\to \Set$. The other conditions ensure that this functor is half-exact, and Corollary \ref{dgcohoweak} implies that the spaces
$\{K(n):= \Spf(k \oplus k[-n]\eps)\}_{n \in \Z}$ are right adequate, so $DG_{\Z}\Sp$ satisfies the conditions of
Heller's Theorem (\cite{heller} Theorem 1.3).
\end{proof}

\begin{corollary}\label{manho}
Let $\L =k$. Given an $L_{\infty}$-algebra $V$, and $A \in dg_{\Z}\C_k$, there is a canonical isomorphism
$$
\ddef(V) \cong \Hom_{\Ho(DG_{\Z}\Sp)}(\Spf A, \mc(V)).
$$
\end{corollary}
\begin{proof}
The key observation is that functors satisfying the conditions of Lemma \ref{manrep} induce  deformation functors (in the sense of Definition \ref{mandef}) on restriction to $\Ho(dg_{\Z}\C_k)$.
%

Now take  a deformation functor $G$ and a natural transformation $\eta:\mc(V) \to G$. Since $G$ is a deformation functor, for any   path object $PA \in dg_{\Z}\C_k$ of an object $A$, the maps $G(PA) \Rightarrow G(A)$ are isomorphisms, and thus $\eta_A$ factors through the coequaliser of $\mc(V)(PA) \Rightarrow \mc(V)(A)$, which is $FA$.  

Therefore $F$ is universal among deformation functors under $\mc(V)$, so $F= \mc(V)^+=\ddef(V)$, as required.
\end{proof}

\begin{remark}\label{manchar}
The proof of Corollary \ref{manho} shows that any homotopy representable functor $F= X[-]$ on $dg_{\Z}\hat{\C}_{\L}$ restricts to a deformation functor on  $dg_{\Z}\C_{\L}$, and that this functor is ``geometric'' in the sense of \cite{Man2} \S7, since 
$X[-] =\Hom_{DG_{\Z}\Sp}(\Spf -, X)^+$.

The deformation functors $F:dg_{\Z}\C_{\L}\to \Set$ of Definition \ref{mandef} satisfy the restriction to  $dg_{\Z}\C_{\L}$ of the conditions of Lemma \ref{manrep}, whenever they make sense. However,   lifting deformation functors to the functors of Lemma \ref{manrep} is not straightforward; as Andrey Lazarev has pointed out to the author, this question is a form of Adams representability. 
\end{remark}

\subsection{The total functor}

\begin{definition}
Define the total complex  functor $\Tot^{\Pi}:DGdg\hat{\C}_{\L}\to  dg_{\Z}\hat{\C}_{\L}$ by the formula of Definition \ref{totprod}, with product coming from that on $R$.
\end{definition}

\begin{theorem}\label{totequiv}
$\Tot^{\Pi}:DGdg\hat{\C}_{\L}\to  dg_{\Z}\hat{\C}_{\L}$ is a right Quillen equivalence.
\end{theorem}
\begin{proof}
$\Tot^{\Pi}$ is clearly right Quillen. Denote its left adjoint by $\Tot^{\Pi*}$. We need to show that for all $S \in DGdg\hat{\C}_{\L}$ the co-unit $\Tot^{\Pi*}Q\Tot^{\Pi}S \to S$ is a weak equivalence for a cofibrant approximation $Q\Tot^{\Pi}S\to \Tot^{\Pi}S $, and that for all   cofibrant $R \in dg_{\Z}\hat{\C}_{\L}$ the unit $R \to \Tot^{\Pi}\Tot^{\Pi*}R$ is a  weak equivalence. 

Since weak equivalences in both categories are determined by cohomology groups (Corollary \ref{weak} and Corollary \ref{dgcohoweak}), it suffices to show that there are canonical isomorphisms
$$
\H^*(\Spf(\Tot^{\Pi*}R)) \cong \H^*(\Spf(R)),\quad \H^*(\Spf(\Tot^{\Pi} S)) \cong \H^*(\Spf(S)).
$$

For the first, observe that $\Tot^{\Pi} K(n)= K(n)$, so
$$
\H^n(\Spf(\Tot^{\Pi*}R))= [\Tot^{\Pi*}R,K(n)] =[\bL\Tot^{\Pi*}R,K(n)]
=[R,\Tot^{\Pi} K(n)]=\H^n(\Spf(R)).
$$ 

For the second, begin by noting that the comparison is unchanged if we replace $S$ by a cofibrant approximation. In $DGdg\hat{\C}_{\L}$, every cofibrant object is free as a pro-Artinian bigraded algebra (although for our purposes, we need only observe that any object of the form $D^*T$, for $T \in cdg\hat{\C}_{\L}$ cofibrant, must be free). Therefore $\Tot^{\Pi}S$ is free as a pro-Artinian graded algebra. If $\cot(S):= \m(S)/(\m(S)^2+\mu)$ and $t(S):= \cot(S)^{\vee}$, then Proposition \ref{totcoho} gives an isomorphism
$$
\H^*(\Spf(S))\cong \H^*(\Tot t(S)).
$$
But $\Tot t(S)= t(\Tot^{\Pi}S)$, and by Lemma \ref{dgcohocalc}, 
$$
\H^*(\Spf(\Tot^{\Pi}S )) \cong \H^*(t(\Tot^{\Pi}S)),
$$
since  $\Tot^{\Pi}S$ is free, hence cofibrant.
\end{proof}

\begin{corollary}\label{allequiv}
Whenever $k$ has characteristic $0$, the categories $\Ho(sc\Sp)$ and $\Ho(DG_{\Z}\Sp)$ are canonically equivalent.
\end{corollary}
\begin{proof}
We have the following chain of left Quillen equivalences:
$$
sc\Sp \xra{\Spf N} sDG\Sp \xla{\Spf D} dgDG\Sp \xra{\Spf \Tot}  DG_{\Z}\Sp,
$$
by Theorems \ref{nequiv}, \ref{dequiv} \and \ref{totequiv}.
\end{proof}

\subsection{Differential $\Z$-graded Lie algebras}

For the purposes of this section, assume that $\L=k$. 

\begin{lemma}\label{mccoho}
The functor $\mc:DG_{\Z}\LA \to DG_{\Z}\Sp$ of Definition \ref{mcdef} is right Quillen. Its left adjoint $\cL$ is given by
$$
\cL(\Spf A) = \cL_q(A^{\vee}),
$$ 
for $\cL_q$ as in Definition \ref{clqdef}.

There are canonical isomorphisms $\H^n(\mc(L))\cong \H^{n+1}(L)$, for all $n\in \Z$, $L \in DG_{\Z}\LA$.
\end{lemma}
\begin{proof}
Immediate.
\end{proof}

We wish to show that $\mc$ is a right Quillen equivalence. To do this, it will suffice to show that there are canonical isomorphisms $\H^n(\cL(X)) \cong \H^{n-1}(X)$, as the unit and co-unit of the adjunction will then be weak equivalences. Our proof will be based on \cite{QRat} Proposition B.6.1, but we need to take more care, since trivial fibration in $dg_{\Z}\hat{\C}_{k}$ is a more restrictive notion than acyclic surjection.

\begin{definition}
Given $L \in DG_{\Z}\LA$, $X \in DG_{\Z}\Sp$, and $\omega \in \mc(L)(X)$, define the total space $E(\omega) \in DG_{\Z}\Sp$ as in \cite{QRat} Proposition B.5.3. There is an isomorphism of graded algebras $O(E(\omega))= O(X)[[L^{\vee}]]$.
\end{definition}


\begin{lemma}\label{Ggprops}
There is a canonical fibration $p_{\omega}:E(\omega) \to X$ in $DG_{\Z}\Sp$. The group space $\exp(L) \in DG_{\Z}\Sp$ given by $\exp(L)(A):= \exp(\z^0\Tot(L\ten \m(A)))$ has a canonical action on $E(\omega)$, with respect to which it is principal bundle over $X$. In particular, the fibre of $p_{\omega}$ over $\Spf k$ is isomorphic to $\exp(L)$.
\end{lemma}
\begin{proof}
It is immediate that $p_{\omega}$ is a fibration, since the associated map of graded algebras is free. The $L$-module structure of \cite{QRat} \S B.5 integrates to give the $\exp(L)$ action. The fibre over $\Spf k$ is $E(0)$, for $0 \in \mc(L)(k)$, which is easily seen to be isomorphic to $L$. 
\end{proof}

\begin{proposition}
For any space $X \in DG_{\Z}\Sp$, the total space $E(\eta(X))$, associated to the unit $\eta(X) \in \mc(\cL(X))(X)$ of the adjunction $\cL\dashv \mc$, is contractible.
\end{proposition}
\begin{proof}
We need to show that $\Spf k \to E(\eta(X))$ is a weak equivalence. By expressing $O(X)\to k$ as a composition of small extensions, it suffices to show that for any small extension $A \to B$ in $dg_{\Z}\C_{k}$, the map $E(\eta(\Spf B)) \to E(\eta(\Spf A))$ is a weak equivalence.

Now, the proof of \cite{QRat} Proposition B.6.1 shows that as a graded coalgebra,
$$
O(E(\eta(\Spf A)))^{\vee} \cong A^{\vee}\ten T(\m(A)^{\vee}[1]),
$$
where $T(V)$ denotes the free tensor algebra on generators $V$, given the coproduct $\Delta(v) =v\ten 1 +1\ten v$. If we write $T_n(V):= \bigoplus_{m \le n} V^{\ten m}$, then  we may define an increasing filtration of sub-DG-coalgebras
 by
$$
F_nO(E(\eta(\Spf A)))^{\vee}:=(\m(A)^{\vee}\ten T_{n-1}(\m(A)^{\vee}[1])) \oplus (k\ten T_n(\m(A)^{\vee}[1])).
$$

Let $U_n(A)$ be the dual of this, so $O(E(\eta(\Spf A)))= \Lim U_n(A)$. It will suffice to show that for all $n$, 
$
f_n:U_n(A) \to U_n(B)
$
is a trivial fibration. We now proceed by induction. If $f_n$ is a trivial fibration, then so is
$$
U_{n}(A)\by_{U_n(B)}U_{n+1}(B) \to U_{n+1}(B),
$$
so it suffices to show that
$$
U_{n+1}(A) \to U_{n}(A)\by_{U_n(B)}U_{n+1}(B)
$$
is a trivial fibration. The kernel $J$ of this map is just 
$$
(I\ten I[1]^{\ten n}) \by (k \ten I[1]^{\ten (n+1)})\cong (k[-1] \oplus k)\ten( I[1]^{\ten n}),
$$
which is acyclic, with $\m(U_{n+1}(A))\cdot J=0$, so this is an acyclic small extension, and hence a trivial fibration.
\end{proof}

\begin{corollary}\label{lcoho}
For all $X \in DG_{\Z}\Sp$, there are canonical isomorphisms $\H^n(\cL(X)) \cong \H^{n-1}(X)$.
\end{corollary}
\begin{proof}
Consider the fibration $\exp(L) \to E(\eta(X)) \xra{p_{\eta}} X$. Since $p_{\eta}$ is a fibration,  $\exp(L)$ is the homotopy fibre, and we have a long exact sequence
$$
\ldots \to\H^{-1}(X) \to \H^0(\exp(L)) \to \H^0( E(\eta(X))) \to \H^0(X) \to \ldots.
$$
However, $E(\eta(X))$ is contractible, so $\H^*(E(\eta(X)))=0$. Since $\H^*(\exp(L))=\H^*(L)$, this gives  $ \H^{n-1}(X) \cong \H^n(\cL(X))$, as required.
\end{proof}

\begin{theorem}\label{mcequiv}
The functor $\mc: DG_{\Z}\LA \to DG_{\Z}\Sp$ is a right Quillen equivalence.
\end{theorem}
\begin{proof}
With the same reasoning as Theorem \ref{totequiv}, this follows from Lemma \ref{mccoho} and Corollary \ref{lcoho}.
\end{proof}

\begin{corollary}\label{cfhin}
For the model category  $DG\mathrm{CU}(k)$ of DG coalgebras  from Lemma \ref{hinmodel}, the equivalence $\iota\co DG_{\Z}\Sp \to DG\mathrm{CU}(k)$ of categories (given by $\Spf A \mapsto A^{\vee}$) identifies the respective model structures. In particular, this implies that weak equivalences in $DG\mathrm{CU}(k)$ between SHLAs are precisely the tangent quasi-isomorphisms (Definition \ref{tqis}) of \cite{Kon}.
\end{corollary}
\begin{proof}
Since  $\iota$ is an equivalence, it is a left adjoint. To see that it is a left Quillen equivalence, observe that for the left  Quillen functor $\cL_q: DG\mathrm{CU}(k)\to DG_{\Z}\LA$ of Definition \ref{clqdef}, we have $\cL=\cL_q\circ \iota$. Since $\cL$ and $\cL_q$ are both Quillen equivalences, $\iota$ must also be so. Now, all objects of $DG_{\Z}\Sp$ are cofibrant, so $\iota$ preserves and reflects weak equivalences, and $\iota$  preserves and reflects cofibrations by definition. Since the model categories have the same cofibrations and weak equivalences, they must also have the same fibrations (by the right lifting property).

For the final statement, just apply Proposition \ref{tqiscor}.
\end{proof}

\begin{corollary}\label{mcallequiv}
Whenever $\L=k$, a field of characteristic $0$, the categories $\Ho(sc\Sp)$, $\Ho(sDG\Sp)$ and $\Ho(DG_{\Z}\LA)$ are canonically equivalent.
\end{corollary}
\begin{proof}
Combine Corollary \ref{allequiv} with Theorem \ref{mcequiv}.
\end{proof}

\begin{proposition}\label{hinnerveprop}
The functor $\Ho(DG_{\Z}\LA) \simeq \cS'$ given by combining Corollary \ref{mcallequiv} with Theorem \ref{dgschrep} is equivalent to Hinich's simplicial nerve functor $\Sigma$ (see Definition \ref{hinnerve}).
\end{proposition}
\begin{proof}
Take $L \in DG_{\Z}\LA$, corresponding under Corollary \ref{mcallequiv} to a fibrant  object $X \in sDG\Sp$. Take $B \in dg\C_k$, and note that
$$
X(B)= \HHom_{sDG\Sp}(\Spf B, X).
$$
Since $\Spf B$ is cofibrant and $X$ is fibrant, this is weakly equivalent to $\bR \Map_{sDG\Sp}(\Spf B, X)$, for $\bR\Map$ as in  Proposition \ref{hinnerve2}. We may regard $B$ as an object in $dg_{\Z} \C_k$, and  the equivalences of Corollary \ref{allequiv} send $\Spf B$ to itself in $DG_{\Z}\Sp$. 

Since $ \bR \Map$ is invariant under Quillen equivalences, this means that
$$
X(B) \simeq \bR \Map_{DG_{\Z}\Sp}( \Spf B, \mc(L)).
$$
Now, as in Proposition \ref{hinnerve2}, $[n] \mapsto \mc(L \ten \sA_n)$ is a fibrant simplicial resolution of $\mc(L)$, so $X(B)$ is weakly equivalent to the simplicial set given by
$$
[n] \mapsto \Hom_{DG_{\Z}\Sp}(\Spf B, \mc(L\ten \sA_n)) = \mc(L\ten \sA_n)(B) = \Sigma(L)(B)_n.
$$

Thus $X(B) \simeq \Sigma(L)(B)$, as required.
\end{proof}

Now, we are in a position to answer  Question 4.6 posed  in \cite{toenseattle} 4.4.2. Take a geometric $D^-$-stack $F$ over $k$ (in the sense of Remark \ref{dgtoen})
with a $k$-valued point $x$, let $\Omega_xF$ be   the loop space of $F$ at $x$, and $L_x(F)$ its tangent space at $x$. [loc. cit.] then asserts that $L_x(F)$ is ``a Lie algebra (or at least an $L_{\infty}$-algebra)", and asks whether the functor $F_x: dg\C_k \to \bS$ (defined analogously to Definition \ref{toenstalk}) is weakly equivalent to Hinich's simplicial nerve $\Sigma(L_x(F))$. 

\begin{proposition}\label{seattle}
In the scenario above, $L_x(F)$ has the natural structure of an $L_{\infty}$-algebra, and the functors $F_x$ and  $\Sigma(\cL_q C_x(F))$ are weakly equivalent, where $C_x(F)$ is the dg coalgebra generated by $L_x(F)$.
\end{proposition}
\begin{proof}
By Corollary \ref{mcallequiv} and Proposition \ref{hinnerveprop}, there exists a $\Z$-graded DGLA $L$ (unique up to quasi-isomorphism), such that $F_x \simeq \Sigma(L)$. Lemma \ref{Ggprops} implies that  $\mc(L)$ is a classifying space for $\exp(L)$ in $DG_{\Z}\Sp$, so $\exp(L)$ is the loop space of $\mc(L)$. Since loop space constructions are preserved by Quillen equivalences of pointed model categories, $\Omega F_x$ corresponds under the equivalence of Corollary \ref{allequiv} to $\exp(L) \in DG_{\Z}\Sp$.

Now, the simplicial complex $\exp(L\ten \sA_{\bt})$ (given in level $n$ by
 $ \exp(L \ten \sA_n)$) is a fibrant simplicial resolution for $\exp(L)$ in $DG_{\Z}\Sp$, so (similarly to Proposition \ref{hinnerveprop}),  $\Omega F_x$ is weakly equivalent to the functor $\exp(L \ten \sA_{\bt}):  dg\C_k \to \bS$. 
 
Therefore, for $t:\cS' \to sDG\Vect$ as in  Remark \ref{ntanfdef2}, there is an equivalence 
$$
L_x(F):= \Tot N^s t(\Omega F_x) \simeq \Tot N^s t\exp((L \ten \sA_{\bt}))
$$
of total tangent spaces in $DG_{\Z}\Vect$.

Now, the tangent space of $\exp(L \ten \sA_{\bt})$ is given by
$$
t(\exp(L \ten \sA_n))= \sigma^{\ge 0}(L \ten \sA_n),
$$
where $\sigma^{\ge 0}$ denotes brutal truncation in non-negative degrees. This has the natural structure of a simplicial complex of DGLAs, so applying the simplicial normalisation functor $N^s$ makes $N^st(\exp(L \ten \sA_{\bt}))$ into a bigraded DGLA (using the Eilenberg-Zilber shuffle product as in \cite{QRat}).

Therefore the cochain complex $\Tot N^st(\exp(L \ten \sA_{\bt}))$ is a DGLA, and is canonically quasi-isomorphic to $L_x(F)$. This gives $L_x(F)$ an $L_{\infty}$-structure, unique up to $L_{\infty}$-equivalence. Thus the dg coalgebra $C_x(F)$ generated by $L_x(F)$ is equivalent to $\C_q\Tot N^st(\exp(L \ten \sA_{\bt})) $, and
$$
\Sigma(\cL_qC_x(F)) \simeq \Sigma(\Tot N^st(\exp(L \ten \sA_{\bt}))).
$$

As in \cite{HinSch}, integration  gives  a quasi-isomorphism
$$
\int:\Tot (N^s\sA_{\bt})\to k 
$$
of DG algebras. Since $\Tot N^st(\exp(L \ten \sA_{\bt}))$ is a sub-DGLA of $L \ten \Tot (N^s\sA_{\bt})$, this gives us a morphism
$$
\theta:\Tot N^st(\exp(L \ten \sA_{\bt})) \to L
$$
of DGLAs.

Since $F_x$ is equivalent to $\mc(L)$ via  Corollary \ref{mcallequiv}, we have $\H^i(F_x) \cong  \H^{i+1}(L)$, so $\H^i(\Omega F_x) \cong \H^i(L)$. Therefore $\theta$ is a quasi-isomorphism, so
$$
\Sigma(\Tot N^st(\exp(L \ten \sA_{\bt}))) \simeq \Sigma(L),
$$
which in turn is equivalent to $F_x$ by  Proposition \ref{hinnerveprop}.

Thus we have shown that 
$$
\Sigma(\cL_qC_x(F)) \simeq F_x,
$$
as required.
\end{proof}

\section{Operations on  cohomology}\label{sopsH}

\subsection{Homology of symmetric products}

\begin{definition}
Recall that $V \in cs\widehat{\FD\Vect}$ is said to be quasi-smooth if  $\H^n(  N_c V_i)=0$ for all $n,i\ge 0$ and $\H_i(N_cV)^n=0$ for all $i>0$ and $n>0$.
\end{definition}

\begin{definition}\label{corner}
Given $V \in cs\widehat{\FD\Vect}$ quasi-smooth, define a cochain complex $N_c\lrcorner V$ in $sDG\widehat{\FD\Vect}$ by:
$$
(N_c\lrcorner V)^n:= \left\{ \begin{matrix} V^0 & n=0\\ \H_0(N_c^nV) & n >0, \end{matrix} \right. 
$$
then set $\lrcorner V:= N_c^{-1}N_c\lrcorner V \in cs\widehat{\FD\Vect}$.
\end{definition}

\begin{lemma}
For $V \in cs\widehat{\FD\Vect}$ quasi-smooth, the projection map $q:V \to \lrcorner V$ is a Reedy weak equivalence, i.e. for all $n$, $q^n: V^n \to (\lrcorner V)^n$ is a weak equivalence in  $s\widehat{\FD\Vect}$.
\end{lemma}

\begin{definition}
For  $V \in \widehat{\FD\Vect}$, define $\Symm(V)$ to be the free power series algebra $k[[V]]$ on generators $V$.
\end{definition}

\begin{lemma}
For $V \in cs\widehat{\FD\Vect}$ quasi-smooth, the projection map $\Symm(q):\Symm(V) \to \Symm(\lrcorner V)$ is a Reedy weak equivalence.
\end{lemma}
\begin{proof}
This follows from \cite{dold}, which shows that $\Symm$ preserves weak equivalences. 
\end{proof}

\begin{definition}
Given a positively graded pro-finite-dimensional $k$-vector space $V_*$, we define 
$$
\fS(V)_*:= \H_*(\Symm ((N^s)^{-1}V_*)).
$$
Given a non-positively graded pro-finite-dimensional $k$-vector space $V_*$, write  $\breve{V}$ for the graded vector space $\breve{V}^i:=U_{-i}$, and set 
$$
\fS(V)_*:= \H^{-*}(\Symm (N_c^{-1}\breve{V}^*)).
$$
Finally, for a $\Z$-graded vector pro-finite-dimensional $k$-vector space $V_*$, set
$$
\fS(V)_n:=\prod_{i+j=n} \fS(V_{>0})_i\ten \fS(V_{\le 0})_j \in \widehat{\FD\Vect}.
$$
\end{definition}

\begin{proposition}
For $V \in cs\widehat{\FD\Vect}$ quasi-smooth, $\H_*(\Tot^{\Pi}N\Symm(V))\cong \fS(\H_*(\Tot^{\Pi}NV))$, for $\Tot^{\Pi}$ as in Definition \ref{totprod}.
\end{proposition}
\begin{proof}
Consider the spectral sequence
$$
E^2_{a,-b}=\H^b(\H_a(N\Symm(V))) \abuts \H_{a-b}(\Tot^{\Pi}N\Symm(V)).
$$  
Since $q:V \to \lrcorner V$ is a Reedy weak equivalence, it gives an isomorphism on the $E^2$ term of the respective spectral sequences, and thus we get
$$
\H_{*}(\Tot^{\Pi}N\Symm(V))\cong \H_{*}(\Tot^{\Pi}N\Symm(\lrcorner V))
$$
at the limit.

We may now choose a decomposition $\lrcorner V= \H_0(V) \oplus W$, and write $U=\H_0(V)$. Thus $U$ is a cosimplicial complex, and $W$ a simplicial complex.
As $\Symm (U \oplus W) = \Symm(U) \ten\Symm(W) $, the simplicial and cosimplicial Eilenberg-Zilber theorems together show that
$$
\H_n(\Tot^{\Pi}N\Symm(V) )\cong \prod_{i+j=n}\H_i(\Tot^{\Pi}N\Symm(U))\hat{\ten} \H_j(\Tot^{\Pi}N\Symm(W)).
$$

Now, $N\Symm(W)$ is just the chain  complex $N^s\Symm(W)$ concentrated in cochain degree $0$, and $N\Symm(U)$ is just the cochain  complex $N_c\Symm(U)$ concentrated in chain degree $0$, so
$$
\H_*(\Tot^{\Pi}N\Symm(W))=\H_*(\Symm(W)), \quad \H_*(\Tot^{\Pi}N\Symm(U))=\H^{-*}(\Symm(U)).
$$

Finally, the results of \cite{Milgram} and \cite{Smirnov} show that $\Symm$ preserves weak equivalences of both simplicial and cosimplicial complexes, so
\begin{eqnarray*}
\H_*(\Symm(W))&=& \fS(\H_*(W))=\fS(\H_{>0}(\Tot^{\Pi}NV)),\\ 
\H_*(\Symm(U))&=& \fS(\H_*(U))=\fS(\H_{\le 0}(\Tot^{\Pi}NV)),
\end{eqnarray*}
as required.
\end{proof}

\begin{remark}\label{symmj}
If $p$ is the characteristic of $k$, then  for $j<p$ (or $p=0$) note that  $\fS^j=\Symm^j$, the graded symmetric power. In general, $\fS$ is very complicated, and  has been computed in \cite{Milgram} and \cite{Smirnov}.  In the notation of \cite{Milgram} Theorem 4.2, for $n>0$, $\fS(k[-n])= \sR(A(\Z,n); k)$. In the notation of \cite{Smirnov} Theorems 1 and 2, $\fS(k[n]) =\H^*(\underline{\sE}_n)$, $\fS(k[-n]) =\H^*(\overline{\sE}_n)^{\vee}$.
\end{remark}

\subsection{The Adams spectral sequence}\label{adams}

For any quasi-smooth left-exact functor $F:s\C_{\L}\to \bS$, the cohomology groups $\H^*(F)$ form a $\Z$-graded vector space. Let $F$ be pro-represented by $R$, and write $\H_i(\cot R)$ for the  pro-finite-dimensional vector space dual to $\H^i(F)$. 

Now, there is a decreasing filtration on $R$ given by $F^i R= \m(R)^i + \mu \m(R)^{i-2}$, and since $F$ is quasi-smooth, 
\begin{eqnarray*}
 \Gr^0 R &=& k\\
\Gr^1 R  &=& \cot R\\
\Gr^a R  &=& \Symm^a \cot R \oplus \bigoplus_{r=1}^{a-1} (\mu^r/\mu^{r+1})\ten \Symm^{a-1-r}\cot R
\end{eqnarray*}
for $a>1$, so that
\begin{eqnarray*}
\H_{*}( \Tot \Gr^0 R)&=& k\\
\H_{*}( \Tot Gr^1 R)&=& H^*(F)^{\vee}\\
\H_{*}( \Tot \Gr^a R) &=& \fS^a H^*(F)^{\vee} \oplus \bigoplus_{r=1}^{a-1} (\mu^r/\mu^{r+1})\ten \fS^{a-1-r}H^*(F)^{\vee}.
\end{eqnarray*}
There is then a convergent spectral sequence 
$$
E^1_{ab}= \H_{a+b} (\Tot \Gr^{-a}  R) \Rightarrow \Gr^{-a}\H_{a+b}(\Tot R)
$$
of pro-Artinian $\L$-modules,  respecting the multiplicative structure.

Studying this spectral sequence yields universal operations on cohomology.
For instance:
\begin{proposition} 
Let $p$ be the characteristic of $k$. If $p \ne 2$, there is a graded Lie bracket
$$
[-,-]: \H^m\by \H^n \to \H^{m+n+1},
$$
such that $[a,b]= (-1)^{mn+m+n}[b,a]$. For $p \ne 3$, this satisfies the Jacobi identity
$$
[[a,b],c]= [a,[b,c]] +(-1)^{mn+m+n}[b,[a,c]].
$$
\end{proposition}
\begin{proof}
Take $\L=k$, and look at $d^1_{-1,m+n+2}:E^1_{-1, m+n+2} \to E^1_{-2, m+n+2}$. Since $p \ne 2$, by Remark \ref{symmj} we have $\fS^2= \Symm^2$, so $d^1_{-1,m+n+2}$ is dual to an antisymmetric product. For $p\ne 3$,  $\fS^3= \Symm^3$, so  the condition $d^1_{-2,m+n+2}\circ d^1_{-1,m+n+2}=0$ gives the Jacobi identity.
\end{proof}

\subsection{Operations on cohomology}

\begin{definition}
Given a collection $\{X_{\alpha}\}$ of objects of $\Sp$, define $\bigvee X_{\alpha}$ to be the coproduct in $\Sp$ (given by
$O(\bigvee X_{\alpha}):= {\prod}_k O(X_{\alpha})$).  
\end{definition}

Recall the definition of the objects $K(n) \in sc\Sp$ from \S \ref{srepcoho}, which have the property that  
 $\H^n(X)=[ K(n), X]$. The cohomology groups $\H^n$ define a functor on $\Ho(sc\Sp)$, and we have the following observation.

\begin{proposition}\label{cohoops}
The set of natural transformations $\H^{m_1}(X)\by \ldots \by \H^{m_r}(X) \to \H^n(X)$, functorial in $X \in \Ho(sc\Sp)$, is naturally isomorphic to 
$$
\H^n(\bigvee_{i=1}^r K(m_r)). 
$$
\end{proposition}
\begin{proof}
Since $\H^n$ is represented by $K(n)$, this set of natural transformations is just
$$
[K(n),\bigvee_{i=1}^r K(m_i)]=\H^n(\bigvee_{i=1}^r K(m_r)),
$$
as required.
\end{proof}

\begin{corollary}\label{andreq}
If all $m_r \ge 0$, the natural transformations $\H^{m_1}(X)\by \ldots \by \H^{m_r}(X) \to \H^n(X)$ are the same as the natural transformations
$$
\bD_{\L}^{m_1}(R,k)\by \ldots \by \bD_{\L}^{m_r}(R,k)  \to \bD^n_{\L}(R,k)
$$
on Andr\'e-Quillen cohomology groups over $\L$, functorial in $R \in s\C_{\L}$.
\end{corollary}
\begin{proof}
Since all $m_r \ge 0$,  $Z:=\bigvee_{i=1}^r K(m_i)$ is an object of $c\Sp$. Take a weak equivalence $ Z \to Y$ to a quasi-smooth object $Y$ of $c\Sp$, and note that $Z$ is then weakly equivalent in $sc\Sp$ to $\underline{Y}$, which is quasi-smooth.

Observe that $\H^n(\underline{Y})=\H^n(Y)$ for all $n \ge 0$, trivially. Moreover,  $ \underline{Y}(k[\eps])_n= Y(k[\eps])$ for all $n$, so $\H^{-n}(\underline{Y})= \pi_n\underline{Y}(k[\eps])_n=0$ for all $n>0$. 

Since $\H^n(\underline{Y})= \H^n(Z)$, and $\H^n(Y)=\bD_{\L}^n(Z,k)$, the result follows.
\end{proof}

\begin{corollary}
If $\L=k$, a field of characteristic $0$, then the only operations on cohomology are generated by the Lie bracket, subject to the Jacobi identity. 
\end{corollary}
\begin{proof}
$K(n)$ corresponds to $k \oplus k[-n]\eps \in dg_{\Z}\C_k$. By Corollary \ref{lcoho}, we thus have $\H^n(\bigvee_i K(m_i))=\H^{n-1}(\cL(\bigvee_i K(m_i)))$, and $\cL(\bigvee_i K(m_i))$ is the free graded Lie algebra on generators $\bigoplus_i k[-m_i -1]$, with differential $0$. 
\end{proof}

\begin{remarks}
\begin{enumerate}
\item
 In positive characteristic, the operations are much harder to compute, but for characteristic $2$, \cite{goerss2} can be applied to Corollary \ref{andreq} to give the operations on non-negative cohomology groups. 

\item Operations on negative cohomology groups seem much harder to describe exhaustively. Since most deformation problems do not have any cohomology groups below $\H^{-1}$, Corollary \ref{andreq} still gives a fairly full description for many cases.

\item The functor $\fS$  contains  divided $p$th powers in addition to the usual symmetric powers, so the Adams spectral sequence gives several  cohomology operations in addition to the Lie bracket.

\item It seems plausible that in finite characteristic, there should be a notion of differential Artinian $\fS$-algebras, to whose homotopy category $sc\Sp$ should be Quillen equivalent.
Although $\fS$ is not a quadratic operad, the results of \cite{goerss2} suggest that there should be some form of ``Koszul'' dual operad $\fL$, and a result corresponding to  Theorem \ref{mcequiv}, with the cohomology groups being $\fL$-algebras. 
\end{enumerate}
\end{remarks}

If $\L$ is not a field, we have the following:
\begin{lemma}
$$
\H^n((\bigvee_{i=1}^r K(m_r))/\L)=\H^n((\bigvee_{i=1}^r K(m_r))/k)\oplus \bD^n_{\L}(k,k).
$$
\end{lemma}
\begin{proof}
Letting $Z:=\bigvee_{i=1}^r K(m_i)$, the diagram $Z \to \Spf k \to \Spf \L$ gives the long exact sequence 
$$
\ldots \to \H^n(Z/k) \to \H^n(Z/\L) \to \H^n(k/\L)\to \ldots, 
$$ 
but $Z \to \Spf k$ has a section, giving the required splitting. Finally, $\H^n(k/\L)=\bD^n_{\L}(k,k)$, the Andr\'e-Quillen cohomology group, which is $0$ for $n<0$.
\end{proof}

\bibliographystyle{alphanum}
\addcontentsline{toc}{section}{Bibliography}
\bibliography{references}
\end{document}